\documentclass [12pt]{amsart}
\usepackage[utf8]{inputenc}
\usepackage{amsmath}
\usepackage{amssymb}
\usepackage{hyperref}
\usepackage{amsthm}
\usepackage[table]{xcolor}    % loads also »colortbl«
\usepackage{graphicx,tikz}
\usetikzlibrary{calc,patterns}

\usepackage{comment}
\usepackage{amsfonts}
\usepackage[utf8]{inputenc} %% mac coding
\usepackage[english]{babel}
\usepackage{bbm}
\usepackage{mathtools}
\usepackage{enumerate}
\usepackage{tikz-cd}
\usepackage{tikz-3dplot}
\usepackage{mathrsfs}
\usetikzlibrary{arrows}
\usetikzlibrary{shapes,snakes}

\makeatletter
\newtheorem*{rep@theorem}{\rep@title}
\newcommand{\newreptheorem}[2]{%
\newenvironment{rep#1}[1]{%
 \def\rep@title{#2 \ref{##1}}%
 \begin{rep@theorem}}%
 {\end{rep@theorem}}}
\makeatother

\newtheorem{theorem}{Theorem}[section]
\newtheorem{lemma}[theorem]{Lemma}
\newtheorem{proposition}[theorem]{Proposition}
\newtheorem{corollary}[theorem]{Corollary}
\theoremstyle{definition}
\newtheorem{definition}[theorem]{Definition}
\theoremstyle{definition}
\newtheorem{example}[theorem]{Example}
\newtheorem{remark}[theorem]{Remark}

\newreptheorem{theorem}{Theorem}
\newreptheorem{lemma}{Lemma}
\newreptheorem{corollary}{Corollary}
\newreptheorem{proposition}{Proposition}
\newreptheorem{conjecture}{Conjecture}

\newenvironment{claim}[1]{\par\noindent\underline{Claim:}\space#1}{}
\newenvironment{claimproof}[1]{\par\noindent\underline{Proof:}\space#1}{\leavevmode\unskip\penalty9999 \hbox{}\nobreak\hfill\quad\hbox{$\blacksquare$}}

\def\R{\mathbb{R}}

\def\Z{\mathbb{Z}}

\def\<{{\langle}}
\def\>{{\rangle}}
\def\l{{\lambda}}
\def\m{{\mu}}

\def\Bcal{{\mathcal{B}}}

\def\Dcal{{\mathcal{D}}}

\def\Fcal{{\mathcal{F}}}

\def\Fix{{ \operatorname{Fix}}}

\newcommand\parr[1]{^{({#1})}}

\begin{document}

\title{Plabic graphs and zonotopal tilings}
\author{Pavel Galashin}
\email{galashin@mit.edu}

\begin{abstract}
We say that two sets $S,T\subset\{1,2,\dots,n\}$ are {\it chord separated} if there does not exist a cyclically ordered quadruple $a,b,c,d$ of integers satisfying $a,c\in S-T$ and $b,d\in T-S$. This is a weaker version of Leclerc and Zelevinsky's {\it weak separation}. We show that every maximal {\it by inclusion} collection of pairwise chord separated sets is also maximal {\it by size}. Moreover, we prove that such collections are precisely vertex label collections of fine zonotopal tilings of the three-dimensional cyclic zonotope. In our construction, plabic graphs and square moves appear naturally as horizontal sections of zonotopal tilings and their mutations respectively. 
\end{abstract}

\address{Department of Mathematics, Massachusetts Institute of Technology, Cambridge, MA, 02139, USA.}

\date\today
\maketitle

\def\WS{\Dcal}
\def\PC{\mathbf{A}}
\def\VC{ \mathbf{V}}
\def\Zon{\Zcal}
\def\Tiling{\Delta}
\def\PTiling{\Sigma}
\def\TPTiling{\widetilde{\PTiling}}
\def\TPTFamily{\TPTiling_*}
\def\Face{\tau}
\def\v{\mathbf{v}}
\def\x{\mathbf{x}}
\def\Hyp{H}
\def\rk{ \operatorname{rk}}
\def\cork{ \operatorname{cork}}
\def\codim{ \operatorname{codim}}
\def\conv{ \operatorname{Conv}}
\def\span{ \operatorname{Span}}
\def\Vert{ \operatorname{Vert}}
\def\ipar{{(i)}}
\def\epar{{(e)}}
\def\UP{ \operatorname{UP}}
\def\DOWN{ \operatorname{DOWN}}
\def\proj{ \operatorname{proj}}
\def\maxgap{ \operatorname{maxgap}}
\def\IC{ \operatorname{IC}}

\def\WhCl{{\Wcal}}
\def\BlCl{{\Bcal}}
\def\Bound{{\partial}}

\def\Mcal{{\mathcal{M}}}
\def\Mcalu{{\underline{\mathcal{M}}}}
\def\Mcaltilde{\widetilde{\mathcal{M}}}
\def\Lcal{{\mathcal{L}}}
\def\Tcal{{\mathcal{T}}}
\def\Ical{{\mathcal{I}}}
\def\Dcal{{\mathcal{D}}}
\def\Zcal{{\mathcal{Z}}}
\def\Wcal{{\mathcal{W}}}
\def\Bcal{{\mathcal{B}}}

\def\Cu{\underline{C}}
\def\Xu{\underline{X}}
\def\Yu{\underline{Y}}
\def\Zu{\underline{Z}}
\def\Bu{\underline{B}}

\def\Cyclic{{\bf C}}

\newcommand{\reorient}[2]{ _{\overline #1}#2}

% \setcounter{tocdepth}{1}
% \tableofcontents

\section{Introduction}

In 1998, Leclerc and Zelevinsky~\cite{LZ} defined the notion of \emph{weak separation} while studying the $q$-deformation of the coordinate ring of the flag variety. They showed that for two subsets $S$ and $T$ of the set $[n]:=\{1,2,\dots,n\}$, the corresponding quantum flag minors quasicommute if and only if $S$ and $T$ are \emph{weakly separated} which is a certain natural combinatorial condition that we recall below. They raised an exciting \emph{purity conjecture}: every maximal \emph{by inclusion} collection of pairwise weakly separated subsets of $[n]$ is also maximal \emph{by size}. This conjecture has been proven independently by Danilov-Karzanov-Koshevoy~\cite{DKK10,DKK14} and by Oh-Postnikov-Speyer~\cite{OPS}. In particular, Oh, Postnikov, and Speyer showed that maximal by inclusion weakly separated collections of $k$-element sets correspond to reduced \emph{plabic graphs} (see Figure~\ref{fig:strands} for an example). Plabic graphs have been introduced by Postnikov in~\cite{Postnikov} where he used them to construct a certain CW decomposition of the totally nonnegative Grassmannian. Since then it was discovered that plabic graphs have important connections to various fields such as scattering amplitudes~\cite{AHBCGPT}, soliton solutions to the KP equation, and cluster algebras~\cite{KW1,KW2}. It was shown earlier by Scott~\cite{Scott,Scott2} that maximal by inclusion weakly separated collections of subsets form clusters in the cluster algebra structure on the coordinate ring of the Grassmannian.

Let us recall the definition of \emph{weak separation} from~\cite{LZ}. Given two subsets $S,T$ of $[n]$, we say that $S$ \emph{surrounds} $T$ if there do not exist numbers $i<j<k\in [n]$ such that $i,k\in T-S$ and $j\in S-T$. We say that $S$ and $T$ are \emph{strongly separated} if $S$ surrounds $T$ and $T$ surrounds $S$. We say that $S$ and $T$ are \emph{weakly separated} if at least one of the following holds:
\begin{enumerate}
 \item $|S|\leq |T|$ and $S$ surrounds $T$;
 \item $|T|\leq |S|$ and $T$ surrounds $S$.
\end{enumerate}
This definition simplifies considerably when $|S|=|T|$. Namely, let us say that $S$ and $T$ are \emph{chord separated} if  there do not exist numbers $a<b<c<d\in [n]$ such that $a,c\in S-T$ and $b,d\in T-S$ or vice versa. In other words, $S$ and $T$ are chord separated if either $S$ surrounds $T$ or $T$ surrounds $S$. One easily observes that when $|S|=|T|$, the two sets are weakly separated if and only if they are chord separated. However, for general $S$ and $T$ our notion of chord separation is weaker than that of weak separation. 

% 
% \[{n\choose 0}+ {n\choose 1}+{n\choose 2}+{n\choose 3}\]
% elements, and correspond to fine zonotopal tilings of the cyclic zonotope $\Zon_{\Cyclic(n,3)}$. Such zonotopal tilings correspond to the elements of the \emph{higher Bruhat order} $B(n,3)$ of Manin and Shekhtman (see~\cite{MS,VK}) by a result of Ziegler~\cite{Ziegler}. More generally, we have thus obtained a new description of the elements of any higher Bruhat order $B(n,d)$: they are in a one-to-one correspondence with maximal \emph{by size} $\Cyclic(n,d)$-separated collections of subsets of $[n]$. 
We say that a collection $\WS\subset 2^{[n]}$ of subsets of $[n]$ is \emph{weakly separated} if any two sets $S,T\in\WS$ are weakly separated from each other. We can now state the purity phenomenon for weak separation:
\begin{theorem}[see \cite{OPS,DKK10,DKK14}]
\label{thm:purity_known}
 \begin{enumerate}
 \item Every maximal \emph{by inclusion} weakly separated collection $\WS\subset 2^{[n]}$ has size 
 \[{n\choose 2}+{n\choose 1}+{n\choose 0}.\]
 \item\label{item:nchoosek} Every maximal \emph{by inclusion} weakly separated collection $\WS\subset {[n]\choose k}$ has size 
 \[k(n-k)+1.\]
\end{enumerate}
Here ${[n]\choose k}$ denotes the collection of all $k$-element subsets of $[n]$.
\end{theorem}

The second statement shows that every maximal \emph{by inclusion} \underline{chord} separated collection $\WS\subset {[n]\choose k}$ of $k$-element sets is also maximal \emph{by size}. As we have already mentioned, Oh-Postnikov-Speyer showed that such collections are in a one-to-one correspondence with reduced plabic graphs (see Section~\ref{sect:plabic_tilings} for the definition). 

The second part of our story involves zonotopal tilings. Zonotopes are Minkowski sums of segments, and zonotopal tilings of a given zonotope $\Zon$ are polytopal subdivisions of $\Zon$ into smaller zonotopes. We say that a zonotopal tiling is \emph{fine} (also called \emph{tight} in the literature) if all of its zonotopes are parallelotopes. Zonotopal tilings have been a popular subject of research for a long time, and have received much attention during the past two decades in the context of the \emph{generalized Baues problem}, see~\cite{BKS}, \cite{Reiner}, and~\cite[Section~7.2]{Book}. Another famous conjecture, the Extension space conjecture, is related to zonotopal tilings via the celebrated Bohne-Dress theorem~\cite{Bohne} which gives an unexpected correspondence between zonotopal tilings of a given zonotope and one-element liftings of the associated oriented matroid. We are going to be interested in the three-dimensional cyclic zonotope $\Zon(n,3)$ that is associated to the \emph{cyclic vector configuration $\Cyclic(n,3)$}. The endpoints of vectors in $\Cyclic(n,3)$ are the vertices of a convex $n$-gon lying in the affine plane $z=1$, see Figure~\ref{fig:cyclic}. Given a fine zonotopal tiling $\Tiling$ of $\Zon(n,3)$, its vertices are naturally labeled by subsets of $[n]$ (see Figure~\ref{fig:tiling_3d}), and we denote this collection of subsets by $\Vert(\Tiling)\subset 2^{[n]}$.
 
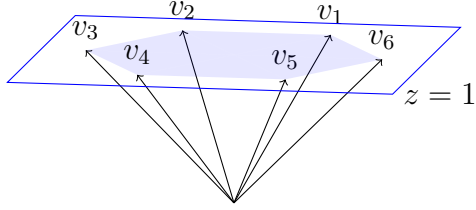
\begin{figure}
\tdplotsetmaincoords{80}{10}
\begin{tikzpicture}[scale=2,tdplot_main_coords]
    \foreach\x in {1,2,...,6}{
      \draw[->] (0,0,0) -- ({cos(\x*60)},{sin(\x*60)},1) node[anchor=south]{$v_\x$};
      \fill[opacity=0.1,blue] (0,0,1) -- ({cos(\x*60)},{sin(\x*60)},1) -- ({cos(\x*60+60)},{sin(\x*60+60)},1)--cycle;
    }
    \def\y{1.3}
    \draw[blue] (\y,\y,1)  -- (\y,-\y,1)node[anchor=west,black]{$z=1$} -- (-\y,-\y,1) -- (-\y,\y,1)  --cycle;
\end{tikzpicture}
\caption{\label{fig:cyclic}The configuration $\Cyclic(6,3)$.} 
\end{figure}

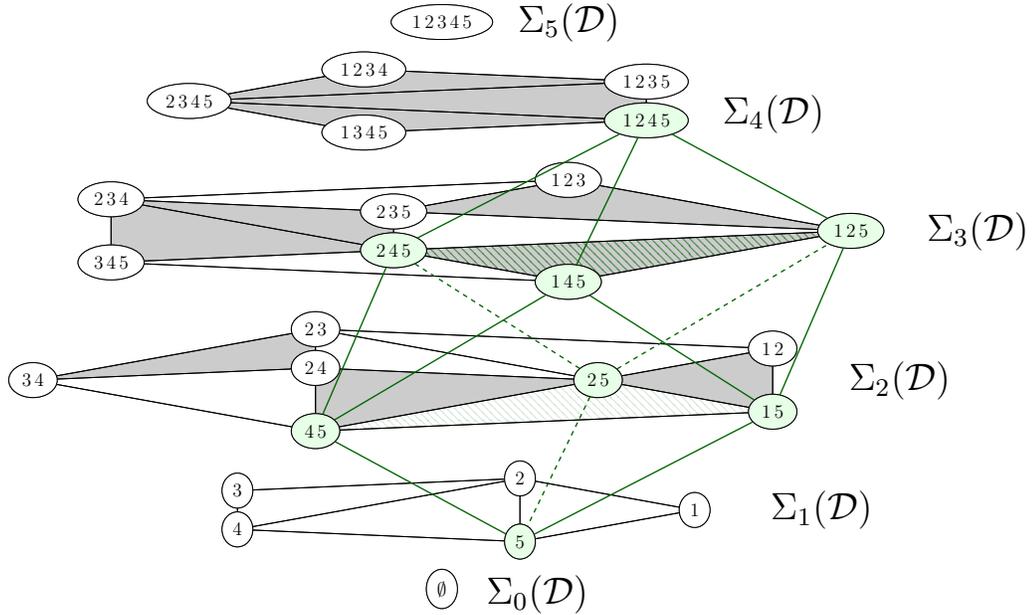
\begin{figure}\scalebox{0.6}{
\begin{tikzpicture}
\node[draw,ellipse,black,fill=white] (node0) at (0.00,0.00) {$\emptyset$};
\node[draw,ellipse,black,fill=white] (node1) at (5.60,1.75) {$1$};
\node[draw,ellipse,black,fill=white] (node2) at (1.73,2.45) {$2$};
\node[draw,ellipse,black,fill=white] (node3) at (-4.53,2.18) {$3$};
\node[draw,ellipse,black,fill=white] (node4) at (-4.53,1.32) {$4$};
\node[draw,ellipse,black,fill=white] (node5) at (1.73,1.05) {$5$};
\node[draw,ellipse,black,fill=white] (node12) at (7.33,5.33) {$1\,2$};
\node[draw,ellipse,black,fill=white] (node15) at (7.33,3.93) {$1\,5$};
\node[draw,ellipse,black,fill=white] (node23) at (-2.80,5.76) {$2\,3$};
\node[draw,ellipse,black,fill=white] (node24) at (-2.80,4.90) {$2\,4$};
\node[draw,ellipse,black,fill=white] (node25) at (3.46,4.63) {$2\,5$};
\node[draw,ellipse,black,fill=white] (node34) at (-9.06,4.63) {$3\,4$};
\node[draw,ellipse,black,fill=white] (node45) at (-2.80,3.50) {$4\,5$};
\node[draw,ellipse,black,fill=white] (node123) at (2.80,9.07) {$1\,2\,3$};
\node[draw,ellipse,black,fill=white] (node125) at (9.06,7.94) {$1\,2\,5$};
\node[draw,ellipse,black,fill=white] (node145) at (2.80,6.81) {$1\,4\,5$};
\node[draw,ellipse,black,fill=white] (node234) at (-7.33,8.64) {$2\,3\,4$};
\node[draw,ellipse,black,fill=white] (node235) at (-1.07,8.37) {$2\,3\,5$};
\node[draw,ellipse,black,fill=white] (node245) at (-1.07,7.51) {$2\,4\,5$};
\node[draw,ellipse,black,fill=white] (node345) at (-7.33,7.24) {$3\,4\,5$};
\node[draw,ellipse,black,fill=white] (node1234) at (-1.73,11.52) {$1\,2\,3\,4$};
\node[draw,ellipse,black,fill=white] (node1235) at (4.53,11.25) {$1\,2\,3\,5$};
\node[draw,ellipse,black,fill=white] (node1245) at (4.53,10.39) {$1\,2\,4\,5$};
\node[draw,ellipse,black,fill=white] (node1345) at (-1.73,10.12) {$1\,3\,4\,5$};
\node[draw,ellipse,black,fill=white] (node2345) at (-5.60,10.82) {$2\,3\,4\,5$};
\node[draw,ellipse,black,fill=white] (node12345) at (-0.00,12.57) {$1\,2\,3\,4\,5$};
\fill [opacity=0.2,black] (node12.center) -- (node15.center) -- (node25.center) -- cycle;
\fill [opacity=0.2,black] (node23.center) -- (node24.center) -- (node34.center) -- cycle;
\fill [opacity=0.2,black] (node24.center) -- (node25.center) -- (node45.center) -- cycle;
\fill [opacity=0.2,black] (node123.center) -- (node125.center) -- (node235.center) -- cycle;
\fill [opacity=0.2,black] (node125.center) -- (node145.center) -- (node245.center) -- cycle;
\fill [opacity=0.2,black] (node234.center) -- (node235.center) -- (node245.center) -- cycle;
\fill [opacity=0.2,black] (node234.center) -- (node245.center) -- (node345.center) -- cycle;
\fill [opacity=0.2,black] (node1234.center) -- (node1235.center) -- (node2345.center) -- cycle;
\fill [opacity=0.2,black] (node1235.center) -- (node1245.center) -- (node2345.center) -- cycle;
\fill [opacity=0.2,black] (node1245.center) -- (node1345.center) -- (node2345.center) -- cycle;

\fill [opacity=0.8,pattern=north west lines, pattern color=black!60!green] (node125.center) -- (node245.center) -- (node145.center) -- cycle;
\fill [opacity=0.4,pattern=north west lines, pattern color=black!60!green] (node15.center) -- (node25.center) -- (node45.center) -- cycle;

\draw[line width=0.04mm,black] (node1) -- (node2);
\draw[line width=0.04mm,black] (node1) -- (node5);
\draw[line width=0.04mm,black] (node2) -- (node1);
\draw[line width=0.04mm,black] (node2) -- (node3);
\draw[line width=0.04mm,black] (node2) -- (node4);
\draw[line width=0.04mm,black] (node2) -- (node5);
\draw[line width=0.04mm,black] (node3) -- (node2);
\draw[line width=0.04mm,black] (node3) -- (node4);
\draw[line width=0.04mm,black] (node4) -- (node2);
\draw[line width=0.04mm,black] (node4) -- (node3);
\draw[line width=0.04mm,black] (node4) -- (node5);
\draw[line width=0.04mm,black] (node5) -- (node1);
\draw[line width=0.04mm,black] (node5) -- (node2);
\draw[line width=0.04mm,black] (node5) -- (node4);
\draw[line width=0.04mm,black] (node12) -- (node15);
\draw[line width=0.04mm,black] (node12) -- (node23);
\draw[line width=0.04mm,black] (node12) -- (node25);
\draw[line width=0.04mm,black] (node15) -- (node12);
\draw[line width=0.04mm,black] (node15) -- (node25);
\draw[line width=0.04mm,black] (node15) -- (node45);
\draw[line width=0.04mm,black] (node23) -- (node12);
\draw[line width=0.04mm,black] (node23) -- (node24);
\draw[line width=0.04mm,black] (node23) -- (node25);
\draw[line width=0.04mm,black] (node23) -- (node34);
\draw[line width=0.04mm,black] (node24) -- (node23);
\draw[line width=0.04mm,black] (node24) -- (node25);
\draw[line width=0.04mm,black] (node24) -- (node34);
\draw[line width=0.04mm,black] (node24) -- (node45);
\draw[line width=0.04mm,black] (node25) -- (node12);
\draw[line width=0.04mm,black] (node25) -- (node15);
\draw[line width=0.04mm,black] (node25) -- (node23);
\draw[line width=0.04mm,black] (node25) -- (node24);
\draw[line width=0.04mm,black] (node25) -- (node45);
\draw[line width=0.04mm,black] (node34) -- (node23);
\draw[line width=0.04mm,black] (node34) -- (node24);
\draw[line width=0.04mm,black] (node34) -- (node45);
\draw[line width=0.04mm,black] (node45) -- (node15);
\draw[line width=0.04mm,black] (node45) -- (node24);
\draw[line width=0.04mm,black] (node45) -- (node25);
\draw[line width=0.04mm,black] (node45) -- (node34);
\draw[line width=0.04mm,black] (node123) -- (node125);
\draw[line width=0.04mm,black] (node123) -- (node234);
\draw[line width=0.04mm,black] (node123) -- (node235);
\draw[line width=0.04mm,black] (node125) -- (node123);
\draw[line width=0.04mm,black] (node125) -- (node145);
\draw[line width=0.04mm,black] (node125) -- (node235);
\draw[line width=0.04mm,black] (node125) -- (node245);
\draw[line width=0.04mm,black] (node145) -- (node125);
\draw[line width=0.04mm,black] (node145) -- (node245);
\draw[line width=0.04mm,black] (node145) -- (node345);
\draw[line width=0.04mm,black] (node234) -- (node123);
\draw[line width=0.04mm,black] (node234) -- (node235);
\draw[line width=0.04mm,black] (node234) -- (node245);
\draw[line width=0.04mm,black] (node234) -- (node345);
\draw[line width=0.04mm,black] (node235) -- (node123);
\draw[line width=0.04mm,black] (node235) -- (node125);
\draw[line width=0.04mm,black] (node235) -- (node234);
\draw[line width=0.04mm,black] (node235) -- (node245);
\draw[line width=0.04mm,black] (node245) -- (node125);
\draw[line width=0.04mm,black] (node245) -- (node145);
\draw[line width=0.04mm,black] (node245) -- (node234);
\draw[line width=0.04mm,black] (node245) -- (node235);
\draw[line width=0.04mm,black] (node245) -- (node345);
\draw[line width=0.04mm,black] (node345) -- (node145);
\draw[line width=0.04mm,black] (node345) -- (node234);
\draw[line width=0.04mm,black] (node345) -- (node245);
\draw[line width=0.04mm,black] (node1234) -- (node1235);
\draw[line width=0.04mm,black] (node1234) -- (node2345);
\draw[line width=0.04mm,black] (node1235) -- (node1234);
\draw[line width=0.04mm,black] (node1235) -- (node1245);
\draw[line width=0.04mm,black] (node1235) -- (node2345);
\draw[line width=0.04mm,black] (node1245) -- (node1235);
\draw[line width=0.04mm,black] (node1245) -- (node1345);
\draw[line width=0.04mm,black] (node1245) -- (node2345);
\draw[line width=0.04mm,black] (node1345) -- (node1245);
\draw[line width=0.04mm,black] (node1345) -- (node2345);
\draw[line width=0.04mm,black] (node2345) -- (node1234);
\draw[line width=0.04mm,black] (node2345) -- (node1235);
\draw[line width=0.04mm,black] (node2345) -- (node1245);
\draw[line width=0.04mm,black] (node2345) -- (node1345);
\node[draw,ellipse,black,fill=white] (node0) at (0.00,0.00) {$\emptyset$};
\node[draw,ellipse,black,fill=white] (node1) at (5.60,1.75) {$1$};
\node[draw,ellipse,black,fill=white] (node2) at (1.73,2.45) {$2$};
\node[draw,ellipse,black,fill=white] (node3) at (-4.53,2.18) {$3$};
\node[draw,ellipse,black,fill=white] (node4) at (-4.53,1.32) {$4$};
\node[draw,ellipse,black,fill=white!90!green] (node5) at (1.73,1.05) {$5$};
\node[draw,ellipse,black,fill=white] (node12) at (7.33,5.33) {$1\,2$};
\node[draw,ellipse,black,fill=white!90!green] (node15) at (7.33,3.93) {$1\,5$};
\node[draw,ellipse,black,fill=white] (node23) at (-2.80,5.76) {$2\,3$};
\node[draw,ellipse,black,fill=white] (node24) at (-2.80,4.90) {$2\,4$};
\node[draw,ellipse,black,fill=white!90!green] (node25) at (3.46,4.63) {$2\,5$};
\node[draw,ellipse,black,fill=white] (node34) at (-9.06,4.63) {$3\,4$};
\node[draw,ellipse,black,fill=white!90!green] (node45) at (-2.80,3.50) {$4\,5$};
\node[draw,ellipse,black,fill=white] (node123) at (2.80,9.07) {$1\,2\,3$};
\node[draw,ellipse,black,fill=white!90!green] (node125) at (9.06,7.94) {$1\,2\,5$};
\node[draw,ellipse,black,fill=white!90!green] (node145) at (2.80,6.81) {$1\,4\,5$};
\node[draw,ellipse,black,fill=white] (node234) at (-7.33,8.64) {$2\,3\,4$};
\node[draw,ellipse,black,fill=white] (node235) at (-1.07,8.37) {$2\,3\,5$};
\node[draw,ellipse,black,fill=white!90!green] (node245) at (-1.07,7.51) {$2\,4\,5$};
\node[draw,ellipse,black,fill=white] (node345) at (-7.33,7.24) {$3\,4\,5$};
\node[draw,ellipse,black,fill=white] (node1234) at (-1.73,11.52) {$1\,2\,3\,4$};
\node[draw,ellipse,black,fill=white] (node1235) at (4.53,11.25) {$1\,2\,3\,5$};
\node[draw,ellipse,black,fill=white!90!green] (node1245) at (4.53,10.39) {$1\,2\,4\,5$};
\node[draw,ellipse,black,fill=white] (node1345) at (-1.73,10.12) {$1\,3\,4\,5$};
\node[draw,ellipse,black,fill=white] (node2345) at (-5.60,10.82) {$2\,3\,4\,5$};
\node[draw,ellipse,black,fill=white] (node12345) at (-0.00,12.57) {$1\,2\,3\,4\,5$};
\node[scale=2,anchor=west] (SIGN_0) at (0.70,-0.12) {$\PTiling_0(\WS)$};
\node[scale=2,anchor=west] (SIGN_1) at (7.00,1.75) {$\PTiling_1(\WS)$};
\node[scale=2,anchor=north west] (SIGN_2) at (8.73,5.33) {$\PTiling_2(\WS)$};
\node[scale=2,anchor=west] (SIGN_3) at (10.46,7.94) {$\PTiling_3(\WS)$};
\node[scale=2,anchor=north west] (SIGN_4) at (5.93,11.25) {$\PTiling_4(\WS)$};
\node[scale=2,anchor=west] (SIGN_5) at (1.40,12.57) {$\PTiling_5(\WS)$};
\draw[line width=0.3mm,black!60!green,opacity=1] (node5) -- (node45) -- (node145) -- (node15) -- (node5);
\draw[line width=0.3mm,black!60!green,opacity=1] (node45) -- (node245) -- (node1245) -- (node125) -- (node15);
\draw[line width=0.3mm,black!60!green,opacity=1] (node145) -- (node1245);
\draw[line width=0.3mm,black!60!green,opacity=1,dashed] (node25) -- (node245);
\draw[line width=0.3mm,black!60!green,opacity=1,dashed] (node25) -- (node125);
\draw[line width=0.3mm,black!60!green,opacity=1,dashed] (node25) -- (node5);

\end{tikzpicture}}

 \caption{\label{fig:tiling_3d} A maximal {by inclusion} chord separated collection $\WS$ of subsets of $\{1,2,3,4,5\}$ and the corresponding zonotopal tiling $\Tiling(\WS)$ of $\Zon(5,3)$. One of the $10$ tiles is shown in green. Each horizontal section $\PTiling_i(\WS)$, $i=0,1,\dots,5,$ is a triangulation of a plabic tiling. }
\end{figure}

We are now ready to state our main result.

\begin{theorem}\label{thm:purity_positive}
\begin{itemize}
 \item Any maximal \emph{by inclusion} chord separated collection $\WS\subset 2^{[n]}$ has \emph{size}
 \[{n\choose 0}+ {n\choose 1}+{n\choose 2}+{n\choose 3}.\]
 \item Given any integer $0\leq k\leq n$, the collection $\WS\cap {[n]\choose k}$ has size $k(n-k)+1$, that is, is maximal \emph{by size} in ${[n]\choose k}$ and thus corresponds to a plabic tiling, which we denote $\PTiling_k(\WS)$.
 \item The map $\Tiling\mapsto\Vert(\Tiling)$ gives a bijection between fine zonotopal tilings of $\Zon(n,3)$ and maximal {by inclusion} chord separated collections $\WS\subset 2^{[n]}$. For any $0\leq k\leq n$, the intersection of $\Tiling$ with the plane $z=k$ in $\R^3$  gives a triangulation of $\PTiling_k(\WS)$ where $\WS=\Vert(\Tiling)$ is the corresponding chord separated collection.
\end{itemize} 
\end{theorem}
The conclusion is, triangulated plabic tilings (or dually, reduced trivalent plabic graphs) are sections of fine zonotopal tilings of the cyclic zonotope $\Zon(n,3)$. In fact, we give a simple compatibility condition on when two triangulated plabic tilings $\PTiling_k$ and $\PTiling_{k+1}$ can appear as the corresponding sections of a fine zonotopal tiling, see Definition~\ref{dfn:compatible}. Theorem~\ref{thm:purity_positive} is illustrated in Figure~\ref{fig:tiling_3d}. Namely, the size of the collection equals $26=1+5+10+10$, the sections are clearly triangulated plabic tilings, and one can easily observe that these sections together form a fine zonotopal tiling $\Tiling$: the eight vertices of any zonotope in $\Tiling$ are of the form $S,Sa,Sb,Sc,Sab,Sac,Sbc,Sabc\in\WS$. Here by $Sa$ we abbreviate $S\cup \{a\}$, etc. The vertices $Sa,Sb,Sc$ form a white triangle in the plabic tiling $\PTiling_{|S|+1}(\WS)$ and the vertices $Sab,Sbc,Sac$ form a black triangle in $\PTiling_{|S|+2}(\WS)$. For example, $\{2,23,12,25,123,235,125,1235\}$ are the vertices of one of the $10$ tiles of $\Tiling$. This tile is highlighted in green in Figure~\ref{fig:tiling_3d}.

\begin{remark}
 In~\cite[Theorem 1.6]{LZ}, Leclerc and Zelevinsky showed that all maximal by inclusion \emph{strongly separated} collections have size
 \[{n\choose 0}+{n\choose 1}+{n\choose 2}\]
 and correspond to pseudoline arrangements, or dually, to zonotopal tilings of the two-dimensional cyclic zonotope. Thus our Theorem~\ref{thm:purity_positive} is a direct three-dimensional analog of their result.
\end{remark}

In Section~\ref{sect:background}, we introduce necessary notation and review the definitions for zonotopal tilings and plabic graphs. We then discuss how plabic graphs, square moves, and strands appear as horizontal sections of their three-dimensional counterparts in Section~\ref{sect:sections}. Finally, we prove Theorem~\ref{thm:purity_positive} in Section~\ref{sect:rk_3}. 
% In Section~\ref{sect:necklaces}, we discuss how our result incorporates most of other known purity phenomena TODO GRASSMANN NECKLACES.

\section*{Acknowledgment}
The author is grateful to Alex Postnikov for his numerous useful suggestions and ideas.

\section{Background}\label{sect:background}

% We denote by $[n]$ the set $\{1,2,\dots,n\}$. For any set $A$ and an integer $k$, by ${A\choose k}$ we denote the collection of all $k$-element subsets of $A$. By $|A|$ we denote the cardinality of $A$. 
Given a set $S$ and an element $e\not\in S$, we denote by $Se$ the union of $S$ and $\{e\}$. The use of $Se$ indicates that $e\not\in S$. By $S-T$ we denote the set-theoretic difference of sets $S$ and $T$, in particular, $S-e=S-\{e\}$. This notation does not imply that $T\subset S$ or $e\in S$. For $a,b\in [n]$, by $[a,b]$ we denote the cyclic interval $\{a,a+1,\dots,b-1,b\}$. The indices are always taken modulo $n$, so if $a>b$ then $[a,b]=[a,n]\cup[1,b]$. Similarly, $(a,b)$ denotes $[a+1,b-1]$ and $(a,b]$ denotes $[a+1,b]$, etc.

\subsection{Zonotopal tilings}

A \emph{signed subset} $X$ of $[n]$ is a pair $(X^+,X^-)$ of disjoint subsets of $[n]$. 
In this case we define its \emph{support} $\Xu:=X^+\sqcup X^-$ and $X^0:=[n]-\Xu$. We denote the set of all signed subsets of $[n]$ by $\{+,-,0\}^{[n]}$ and for each element $i\in [n]$ we write 
\[X_i=\begin{cases}
       +, &\text{ if $i\in X^+$};\\
       -, &\text{ if $i\in X^-$};\\
       0, &\text{ if $i\in X^0$}.
      \end{cases}\]

A \emph{vector configuration} $\VC=(\v_1,\dots,\v_n)$ is a finite subset of $\R^d$.

For a vector configuration $\VC$, define the corresponding \emph{zonotope} $\Zon_\VC$ to be the Minkowski sum 
\[\Zon_\VC:=[0,\v_1]+[0,\v_2]+\dots+[0,\v_n],\]
where the \emph{Minkowski sum} of two subsets $A,B\subset \R^d$ is defined by 
\[A+B=\{a+b\mid a\in A, b\in B\}\subset\R^d.\]

For a signed set $X$, we denote by $\Face_X$ the following zonotope:
% \[\Face_X:=\sum_{i\in X^+} v_i+\sum_{j\not\in\Xu} [0,v_j].\]
\[\Face_X:=\sum_{i\in [n]}\begin{cases}
                           v_i, &\text{ if $i\in X^+$};\\
                           0, &\text{ if $i\in X^-$};\\
                           [0,v_i], &\text{ if $i\in X^0$}.                           
                          \end{cases}
\]

\begin{definition}\label{dfn:tilings_realizable}
A collection $\Tiling$ of signed subsets of $[n]$ is called a \emph{zonotopal tiling} of $\Zon_\VC$ if and only if the following two conditions hold:
\begin{itemize}
 \item $\Zon_\VC=\bigcup\limits_{X\in\Tiling}\, \Face_X$;
 \item\label{item:proper_subd} For any two $X,Y\in\Tiling$, either the intersection $\Face_X\cap\Face_Y$ is empty or there exists $Z\in\Tiling$ such that $\Face_Z$ is a face of $\Face_X$ and of $\Face_Y$, and 
 \[\Face_X\cap\Face_Y=\Face_Z.\]
\end{itemize}
\end{definition}

A zonotopal tiling $\Tiling$ is called \emph{fine} if for every $X\in\Tiling$, $|X^0|\leq d$, that is, if all the top-dimensional tiles are parallelotopes. For a fine zonotopal tiling $\Tiling$, its \emph{set of vertices} is defined as
\[\Vert(\Tiling):=\{X^+\mid X\in\Tiling\text{ such that } \Xu=[n]\}\subset 2^{[n]}.\]

For $1\leq d\leq n$, the \emph{cyclic vector configuration} $\Cyclic(n,d)$ is given by the following collection of vectors $\{\v_1,\dots,\v_n\}\subset \R^d$, where
\[\v_i=(1,x_i,\dots,x_i^{d-1})\]
and $0<x_1<x_2<\dots<x_n\in\R$ are any increasing positive real numbers. 

\subsection{Plabic graphs}\label{sect:plabic_graphs}
A \emph{planar bicolored graph} (\emph{plabic graph} for short) $G$ is a planar graph embedded in a disc so that every non-boundary vertex of $G$ is colored either black or white. Two adjacent vertices are not required to have opposite colors. Plabic graphs were introduced by Postnikov in~\cite{Postnikov} in his study of the totally nonnegative Grassmannian. 

Given a plabic graph $G$, a \emph{strand} in $G$ is a path $(v_1,v_2,\dots,v_r)$ that ``makes a sharp right turn'' at every black vertex and ``makes a sharp left turn'' at every white vertex. More precisely, for any $1<i<r$, the vertices $v_{i-1},v_i,v_{i+1}$ belong to the same face $F$ of $G$ and if $v_i$ is white (resp., black) then $v_{i-1},v_i,v_{i+1}$ go in the counterclockwise (resp., clockwise) direction in the boundary of $F$. Let $b_1,\dots,b_n$ be the boundary vertices of $G$ in a counterclockwise order. Then the \emph{strand permutation} $\pi_G$ of $[n]$ is defined by $\pi_G(i)=j$ whenever the strand that starts at $b_i$ ends at $b_j$. We denote such a strand by $i\to j$. See Figure~\ref{fig:strands}.

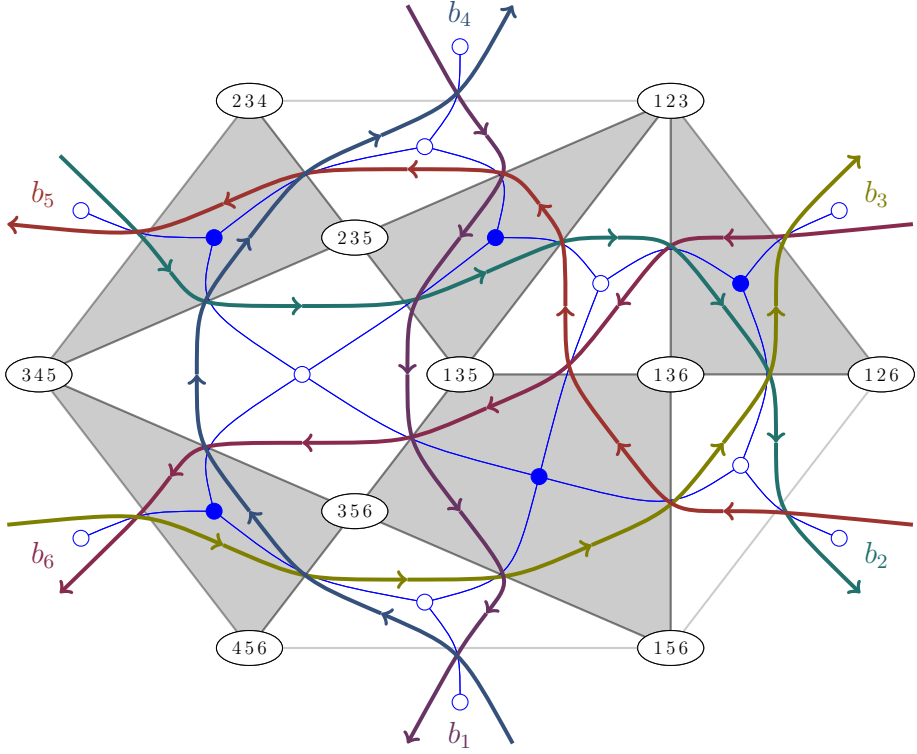
\begin{figure}

\scalebox{1.0}{
\begin{tikzpicture}[every node/.style={scale=0.6}]
\node[draw,ellipse,black,fill=white] (node123) at (2.80,3.64) {$1\,2\,3$};
\node[draw,ellipse,black,fill=white] (node126) at (5.60,0.00) {$1\,2\,6$};
\node[draw,ellipse,black,fill=white] (node135) at (-0.00,0.00) {$1\,3\,5$};
\node[draw,ellipse,black,fill=white] (node136) at (2.80,0.00) {$1\,3\,6$};
\node[draw,ellipse,black,fill=white] (node156) at (2.80,-3.64) {$1\,5\,6$};
\node[draw,ellipse,black,fill=white] (node234) at (-2.80,3.64) {$2\,3\,4$};
\node[draw,ellipse,black,fill=white] (node235) at (-1.40,1.82) {$2\,3\,5$};
\node[draw,ellipse,black,fill=white] (node345) at (-5.60,0.00) {$3\,4\,5$};
\node[draw,ellipse,black,fill=white] (node356) at (-1.40,-1.82) {$3\,5\,6$};
\node[draw,ellipse,black,fill=white] (node456) at (-2.80,-3.64) {$4\,5\,6$};
\draw [opacity=0.2,fill=black,black,thick] (node123.center)-- (node135.center)-- (node235.center) -- cycle;
\draw [opacity=0.2,fill=black,black,thick] (node123.center)-- (node126.center)-- (node136.center) -- cycle;
\draw [opacity=0.2,fill=black,black,thick] (node135.center)-- (node136.center)-- (node156.center)-- (node356.center) -- cycle;
\draw [opacity=0.2,fill=black,black,thick] (node234.center)-- (node235.center)-- (node345.center) -- cycle;
\draw [opacity=0.2,fill=black,black,thick] (node345.center)-- (node356.center)-- (node456.center) -- cycle;
\draw [opacity=0.2,black,thick] (node123.center)-- (node135.center)-- (node235.center) -- cycle;
\draw [opacity=0.2,black,thick] (node123.center)-- (node126.center)-- (node136.center) -- cycle;
\draw [opacity=0.2,black,thick] (node135.center)-- (node136.center)-- (node156.center)-- (node356.center) -- cycle;
\draw [opacity=0.2,black,thick] (node234.center)-- (node235.center)-- (node345.center) -- cycle;
\draw [opacity=0.2,black,thick] (node345.center)-- (node356.center)-- (node456.center) -- cycle;
\draw [opacity=0.2,black,thick] (node123.center)-- (node135.center)-- (node136.center) -- cycle;
\draw [opacity=0.2,black,thick] (node126.center)-- (node136.center)-- (node156.center) -- cycle;
\draw [opacity=0.2,black,thick] (node123.center)-- (node234.center)-- (node235.center) -- cycle;
\draw [opacity=0.2,black,thick] (node135.center)-- (node235.center)-- (node345.center)-- (node356.center) -- cycle;
\draw [opacity=0.2,black,thick] (node156.center)-- (node356.center)-- (node456.center) -- cycle;
\coordinate (node12) at (5.04,2.18);
\coordinate (node13) at (1.87,1.21);
\coordinate (node16) at (3.73,-1.21);
\coordinate (node23) at (-0.47,3.03);
\coordinate (node34) at (-5.04,2.18);
\coordinate (node35) at (-2.10,0.00);
\coordinate (node45) at (-5.04,-2.18);
\coordinate (node56) at (-0.47,-3.03);
\coordinate (node1234) at (0.00,4.36);
\coordinate (node1235) at (0.47,1.82);
\coordinate (node1236) at (3.73,1.21);
\coordinate (node1256) at (5.04,-2.18);
\coordinate (node1356) at (1.05,-1.36);
\coordinate (node1456) at (-0.00,-4.36);
\coordinate (node2345) at (-3.27,1.82);
\coordinate (node3456) at (-3.27,-1.82);
\draw[blue] (5.04,2.18).. controls (4.20,1.82) .. (3.73,1.21);
\draw[blue] (1.87,1.21).. controls (1.40,1.82) .. (0.47,1.82);
\draw[blue] (1.87,1.21).. controls (1.40,0.00) .. (1.05,-1.36);
\draw[blue] (1.87,1.21).. controls (2.80,1.82) .. (3.73,1.21);
\draw[blue] (3.73,-1.21).. controls (4.20,0.00) .. (3.73,1.21);
\draw[blue] (3.73,-1.21).. controls (2.80,-1.82) .. (1.05,-1.36);
\draw[blue] (3.73,-1.21).. controls (4.20,-1.82) .. (5.04,-2.18);
\draw[blue] (-0.47,3.03).. controls (0.00,3.64) .. (0.00,4.36);
\draw[blue] (-0.47,3.03).. controls (-2.10,2.73) .. (-3.27,1.82);
\draw[blue] (-0.47,3.03).. controls (0.70,2.73) .. (0.47,1.82);
\draw[blue] (-5.04,2.18).. controls (-4.20,1.82) .. (-3.27,1.82);
\draw[blue] (-2.10,0.00).. controls (-0.70,0.91) .. (0.47,1.82);
\draw[blue] (-2.10,0.00).. controls (-3.50,0.91) .. (-3.27,1.82);
\draw[blue] (-2.10,0.00).. controls (-3.50,-0.91) .. (-3.27,-1.82);
\draw[blue] (-2.10,0.00).. controls (-0.70,-0.91) .. (1.05,-1.36);
\draw[blue] (-5.04,-2.18).. controls (-4.20,-1.82) .. (-3.27,-1.82);
\draw[blue] (-0.47,-3.03).. controls (0.70,-2.73) .. (1.05,-1.36);
\draw[blue] (-0.47,-3.03).. controls (-2.10,-2.73) .. (-3.27,-1.82);
\draw[blue] (-0.47,-3.03).. controls (-0.00,-3.64) .. (-0.00,-4.36);
\draw[blue] (0.00,4.36).. controls (0.00,3.64) .. (-0.47,3.03);
\draw[blue] (0.47,1.82).. controls (1.40,1.82) .. (1.87,1.21);
\draw[blue] (0.47,1.82).. controls (-0.70,0.91) .. (-2.10,0.00);
\draw[blue] (0.47,1.82).. controls (0.70,2.73) .. (-0.47,3.03);
\draw[blue] (3.73,1.21).. controls (4.20,1.82) .. (5.04,2.18);
\draw[blue] (3.73,1.21).. controls (4.20,0.00) .. (3.73,-1.21);
\draw[blue] (3.73,1.21).. controls (2.80,1.82) .. (1.87,1.21);
\draw[blue] (5.04,-2.18).. controls (4.20,-1.82) .. (3.73,-1.21);
\draw[blue] (1.05,-1.36).. controls (1.40,0.00) .. (1.87,1.21);
\draw[blue] (1.05,-1.36).. controls (2.80,-1.82) .. (3.73,-1.21);
\draw[blue] (1.05,-1.36).. controls (0.70,-2.73) .. (-0.47,-3.03);
\draw[blue] (1.05,-1.36).. controls (-0.70,-0.91) .. (-2.10,0.00);
\draw[blue] (-0.00,-4.36).. controls (-0.00,-3.64) .. (-0.47,-3.03);
\draw[blue] (-3.27,1.82).. controls (-2.10,2.73) .. (-0.47,3.03);
\draw[blue] (-3.27,1.82).. controls (-3.50,0.91) .. (-2.10,0.00);
\draw[blue] (-3.27,1.82).. controls (-4.20,1.82) .. (-5.04,2.18);
\draw[blue] (-3.27,-1.82).. controls (-3.50,-0.91) .. (-2.10,0.00);
\draw[blue] (-3.27,-1.82).. controls (-2.10,-2.73) .. (-0.47,-3.03);
\draw[blue] (-3.27,-1.82).. controls (-4.20,-1.82) .. (-5.04,-2.18);
\draw[->,color={rgb:red,204;green,62;blue,115},line width=1.5pt] (6.02,2.00).. controls (4.20,1.82) .. (3.50,1.82);
\draw[->,color={rgb:red,204;green,62;blue,115},line width=1.5pt] (3.50,1.82).. controls (2.80,1.82) .. (2.10,0.91);
\draw[->,color={rgb:red,204;green,62;blue,115},line width=1.5pt] (2.10,0.91).. controls (1.40,0.00) .. (0.35,-0.45);
\draw[->,color={rgb:red,204;green,62;blue,115},line width=1.5pt] (0.35,-0.45).. controls (-0.70,-0.91) .. (-2.10,-0.91);
\draw[->,color={rgb:red,204;green,62;blue,115},line width=1.5pt] (-2.10,-0.91).. controls (-3.50,-0.91) .. (-3.85,-1.36);
\draw[->,color={rgb:red,204;green,62;blue,115},line width=1.5pt] (-3.85,-1.36).. controls (-4.20,-1.82) .. (-5.32,-2.91);
\node[scale=1.6,color={rgb:red,204;green,62;blue,115}] (bound6) at (-5.54,-2.40) {$b_6$};
\draw[->,color={rgb:red,38;green,129;blue,125},line width=1.5pt] (-5.32,2.91).. controls (-4.20,1.82) .. (-3.85,1.36);
\draw[->,color={rgb:red,38;green,129;blue,125},line width=1.5pt] (-3.85,1.36).. controls (-3.50,0.91) .. (-2.10,0.91);
\draw[->,color={rgb:red,38;green,129;blue,125},line width=1.5pt] (-2.10,0.91).. controls (-0.70,0.91) .. (0.35,1.36);
\draw[->,color={rgb:red,38;green,129;blue,125},line width=1.5pt] (0.35,1.36).. controls (1.40,1.82) .. (2.10,1.82);
\draw[->,color={rgb:red,38;green,129;blue,125},line width=1.5pt] (2.10,1.82).. controls (2.80,1.82) .. (3.50,0.91);
\draw[->,color={rgb:red,38;green,129;blue,125},line width=1.5pt] (3.50,0.91).. controls (4.20,0.00) .. (4.20,-0.91);
\draw[->,color={rgb:red,38;green,129;blue,125},line width=1.5pt] (4.20,-0.91).. controls (4.20,-1.82) .. (5.32,-2.91);
\node[scale=1.6,color={rgb:red,38;green,129;blue,125}] (bound2) at (5.54,-2.40) {$b_2$};
\draw[->,color={rgb:red,200;green,200;blue,0},line width=1.5pt] (-6.02,-2.00).. controls (-4.20,-1.82) .. (-3.15,-2.27);
\draw[->,color={rgb:red,200;green,200;blue,0},line width=1.5pt] (-3.15,-2.27).. controls (-2.10,-2.73) .. (-0.70,-2.73);
\draw[->,color={rgb:red,200;green,200;blue,0},line width=1.5pt] (-0.70,-2.73).. controls (0.70,-2.73) .. (1.75,-2.27);
\draw[->,color={rgb:red,200;green,200;blue,0},line width=1.5pt] (1.75,-2.27).. controls (2.80,-1.82) .. (3.50,-0.91);
\draw[->,color={rgb:red,200;green,200;blue,0},line width=1.5pt] (3.50,-0.91).. controls (4.20,0.00) .. (4.20,0.91);
\draw[->,color={rgb:red,200;green,200;blue,0},line width=1.5pt] (4.20,0.91).. controls (4.20,1.82) .. (5.32,2.91);
\node[scale=1.6,color={rgb:red,200;green,200;blue,0}] (bound3) at (5.54,2.40) {$b_3$};
\draw[->,color={rgb:red,170;green,85;blue,162},line width=1.5pt] (-0.70,4.91).. controls (0.00,3.64) .. (0.35,3.18);
\draw[->,color={rgb:red,170;green,85;blue,162},line width=1.5pt] (0.35,3.18).. controls (0.70,2.73) .. (-0.00,1.82);
\draw[->,color={rgb:red,170;green,85;blue,162},line width=1.5pt] (-0.00,1.82).. controls (-0.70,0.91) .. (-0.70,0.00);
\draw[->,color={rgb:red,170;green,85;blue,162},line width=1.5pt] (-0.70,0.00).. controls (-0.70,-0.91) .. (-0.00,-1.82);
\draw[->,color={rgb:red,170;green,85;blue,162},line width=1.5pt] (-0.00,-1.82).. controls (0.70,-2.73) .. (0.35,-3.18);
\draw[->,color={rgb:red,170;green,85;blue,162},line width=1.5pt] (0.35,-3.18).. controls (-0.00,-3.64) .. (-0.70,-4.91);
\node[scale=1.6,color={rgb:red,170;green,85;blue,162}] (bound1) at (-0.00,-4.80) {$b_1$};
\draw[->,color={rgb:red,239;green,77;blue,71},line width=1.5pt] (6.02,-2.00).. controls (4.20,-1.82) .. (3.50,-1.82);
\draw[->,color={rgb:red,239;green,77;blue,71},line width=1.5pt] (3.50,-1.82).. controls (2.80,-1.82) .. (2.10,-0.91);
\draw[->,color={rgb:red,239;green,77;blue,71},line width=1.5pt] (2.10,-0.91).. controls (1.40,0.00) .. (1.40,0.91);
\draw[->,color={rgb:red,239;green,77;blue,71},line width=1.5pt] (1.40,0.91).. controls (1.40,1.82) .. (1.05,2.27);
\draw[->,color={rgb:red,239;green,77;blue,71},line width=1.5pt] (1.05,2.27).. controls (0.70,2.73) .. (-0.70,2.73);
\draw[->,color={rgb:red,239;green,77;blue,71},line width=1.5pt] (-0.70,2.73).. controls (-2.10,2.73) .. (-3.15,2.27);
\draw[->,color={rgb:red,239;green,77;blue,71},line width=1.5pt] (-3.15,2.27).. controls (-4.20,1.82) .. (-6.02,2.00);
\node[scale=1.6,color={rgb:red,239;green,77;blue,71}] (bound5) at (-5.54,2.40) {$b_5$};
\draw[->,color={rgb:red,76;green,117;blue,179},line width=1.5pt] (0.70,-4.91).. controls (-0.00,-3.64) .. (-1.05,-3.18);
\draw[->,color={rgb:red,76;green,117;blue,179},line width=1.5pt] (-1.05,-3.18).. controls (-2.10,-2.73) .. (-2.80,-1.82);
\draw[->,color={rgb:red,76;green,117;blue,179},line width=1.5pt] (-2.80,-1.82).. controls (-3.50,-0.91) .. (-3.50,0.00);
\draw[->,color={rgb:red,76;green,117;blue,179},line width=1.5pt] (-3.50,0.00).. controls (-3.50,0.91) .. (-2.80,1.82);
\draw[->,color={rgb:red,76;green,117;blue,179},line width=1.5pt] (-2.80,1.82).. controls (-2.10,2.73) .. (-1.05,3.18);
\draw[->,color={rgb:red,76;green,117;blue,179},line width=1.5pt] (-1.05,3.18).. controls (0.00,3.64) .. (0.70,4.91);
\node[scale=1.6,color={rgb:red,76;green,117;blue,179}] (bound4) at (0.00,4.80) {$b_4$};
\draw[blue,fill=white] (node12) circle (3pt);
\draw[blue,fill=white] (node13) circle (3pt);
\draw[blue,fill=white] (node16) circle (3pt);
\draw[blue,fill=white] (node23) circle (3pt);
\draw[blue,fill=white] (node34) circle (3pt);
\draw[blue,fill=white] (node35) circle (3pt);
\draw[blue,fill=white] (node45) circle (3pt);
\draw[blue,fill=white] (node56) circle (3pt);
\draw[blue,fill=white] (node1234) circle (3pt);
\draw[blue,fill=blue] (node1235) circle (3pt);
\draw[blue,fill=blue] (node1236) circle (3pt);
\draw[blue,fill=white] (node1256) circle (3pt);
\draw[blue,fill=blue] (node1356) circle (3pt);
\draw[blue,fill=white] (node1456) circle (3pt);
\draw[blue,fill=blue] (node2345) circle (3pt);
\draw[blue,fill=blue] (node3456) circle (3pt);
\node[draw,ellipse,black,fill=white] (node123) at (2.80,3.64) {$1\,2\,3$};
\node[draw,ellipse,black,fill=white] (node126) at (5.60,0.00) {$1\,2\,6$};
\node[draw,ellipse,black,fill=white] (node135) at (-0.00,0.00) {$1\,3\,5$};
\node[draw,ellipse,black,fill=white] (node136) at (2.80,0.00) {$1\,3\,6$};
\node[draw,ellipse,black,fill=white] (node156) at (2.80,-3.64) {$1\,5\,6$};
\node[draw,ellipse,black,fill=white] (node234) at (-2.80,3.64) {$2\,3\,4$};
\node[draw,ellipse,black,fill=white] (node235) at (-1.40,1.82) {$2\,3\,5$};
\node[draw,ellipse,black,fill=white] (node345) at (-5.60,0.00) {$3\,4\,5$};
\node[draw,ellipse,black,fill=white] (node356) at (-1.40,-1.82) {$3\,5\,6$};
\node[draw,ellipse,black,fill=white] (node456) at (-2.80,-3.64) {$4\,5\,6$};
\end{tikzpicture}}

\caption{\label{fig:strands} A plabic graph $G$ (blue). For each $j\in [n]$, the faces whose labels contain $j$ are to the left of the strand $i\to j$. The collection of face labels is a maximal {by inclusion} weakly separated collection in $[n]\choose k$. The dual of $G$ is a plabic tiling from Section~\ref{sect:plabic_tilings}.}
 
\end{figure}

We say that two strands have an \emph{essential intersection} if there is an edge that they traverse in opposite directions. We say that two strands have a \emph{bad double crossing} if they have essential intersections at two edges $e_1$ and $e_2$ and each of the strands first passes through $e_1$ and then through $e_2$. 

\begin{definition}\label{dfn:plabic}
 A plabic graph $G$ is called \emph{reduced} (see \cite[Theorem~13.2]{Postnikov}) if
\begin{itemize}
 \item there are no closed strands in $G$;
 \item no strand in $G$ has an essential self-intersection;
 \item no two strands in $G$ have a bad double crossing;
 \item if $\pi_G(i)=i$ then $G$ has a boundary leaf attached to $b_i$.
\end{itemize}
\end{definition}

It follows from this definition that a plabic graph cannot have loops or parallel edges. We will additionally assume that it has no non-boundary vertices of degree two.

We will from now on restrict our attention to reduced plabic graphs whose strand permutation sends $i$ to $i+k$ modulo $n$, for all $i\in [n]$. We denote this permutation by $\sigma\parr{k,n}$. In the stratification of the totally nonnegative Grassmannian from~\cite{Postnikov}, $\sigma\parr{k,n}$ corresponds to the top-dimensional cell of $ \operatorname{Gr}_{k,n}^{ \operatorname{tnn}}$. 

\begin{remark}
 According to our Theorem~\ref{thm:purity_positive} combined with the results of \cite{OPS}, reduced plabic graphs are precisely the objects dual to horizontal sections of fine zonotopal tilings of $\Zon(n,3)$. It is quite surprising that all the conditions from Definition~\ref{dfn:plabic} are somehow incorporated in the concept of a zonotopal tiling. 
\end{remark}

As it was shown in~\cite{Postnikov}, all reduced plabic graphs with the same strand permutation are connected by certain moves which we now recall. There are two kinds of moves, \emph{unicolored contraction/uncontraction moves} and \emph{square moves}, see Figure~\ref{fig:moves_plabic}. Note that in order to perform a square move, all four vertices are required to have degree $3$.

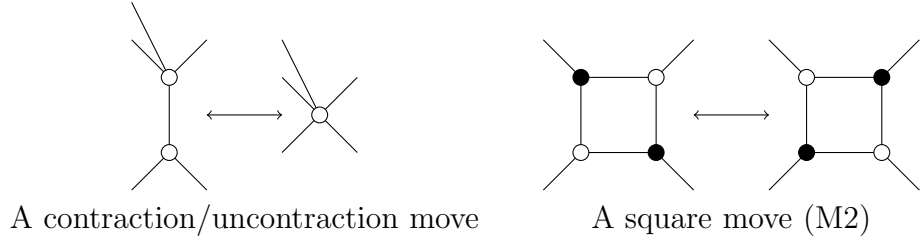
\begin{figure}
\setlength{\tabcolsep}{12pt}
\begin{tabular}{cc}
 \begin{tikzpicture}[scale=0.5]
  \coordinate(a) at (-2,1);
  \coordinate(b) at (-2,-1);
  \coordinate(c) at (2,0);
  \draw (a) -- (b);
  \draw (a) -- (-3,2);
  \draw (a) -- (-3,3);
  \draw (a) -- (-1,2);
  \draw (b) -- (-3,-2);
  \draw (b) -- (-1,-2);
  \draw[<->] (-1,0) -- (1,0);

  \draw (c) -- (1,1);
  \draw (c) -- (1,2);
  \draw (c) -- (3,1);
  \draw (c) -- (1,-1);
  \draw (c) -- (3,-1);
  
  \draw[fill=white] (a) circle (6pt);
  \draw[fill=white] (b) circle (6pt);
  \draw[fill=white] (c) circle (6pt);
 \end{tikzpicture} & 
 \begin{tikzpicture}[scale=0.5]

 \draw (-4,1) -- (-5,2);
 \draw (-4,-1) -- (-5,-2);
 \draw (-2,1) -- (-1,2);
 \draw (-2,-1) -- (-1,-2);
 \draw (4,1) -- (5,2);
 \draw (4,-1) -- (5,-2);
 \draw (2,1) -- (1,2);
 \draw (2,-1) -- (1,-2);
 \draw (-4,1) -- (-2,1) -- (-2,-1) -- (-4,-1) -- cycle;
 \draw (4,1) -- (2,1) -- (2,-1) -- (4,-1) -- cycle;
 
 \draw[<->] (-1,0) -- (1,0);

  \draw[fill=white] (-4,-1) circle (6pt);
  \draw[fill=black] (-4,1) circle (6pt);
  \draw[fill=white] (-2,1) circle (6pt);
  \draw[fill=black] (-2,-1) circle (6pt);
  \draw[fill=white] (4,-1) circle (6pt);
  \draw[fill=black] (4,1) circle (6pt);
  \draw[fill=white] (2,1) circle (6pt);
  \draw[fill=black] (2,-1) circle (6pt);
 
 \end{tikzpicture}\\
 A contraction/uncontraction move & A square move (M2)
 
\end{tabular}
\setlength{\tabcolsep}{6pt}
\caption{\label{fig:moves_plabic}Two types of moves on plabic graphs}
\end{figure}

Using contraction/uncontraction moves, one can always transform a plabic graph into a \emph{trivalent plabic graph} for which all non-boundary vertices have degree $3$. Any two trivalent plabic graphs are connected by trivalent versions of the moves from Figure~\ref{fig:moves_plabic}, namely, by square moves and contraction-uncontraction moves from Figure~\ref{fig:moves_trivalent}. We denote the square move by (M2) and the white (resp., black) trivalent contraction-uncontraction move by (M1) (resp., (M3)).

\begin{theorem}[\cite{Postnikov}]\label{thm:connected}
 All reduced plabic graphs with the same strand permutation are connected by moves from Figure~\ref{fig:moves_plabic}. All trivalent reduced plabic graphs with the same strand permutation are connected by moves (M1)-(M3).
\end{theorem}

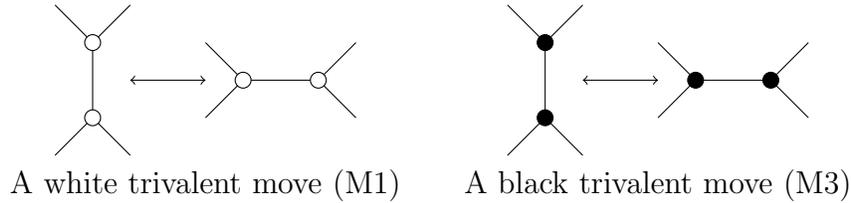
\begin{figure}
\setlength{\tabcolsep}{12pt}
\begin{tabular}{ccc}
 \begin{tikzpicture}[scale=0.5]
  \coordinate(a) at (-2,1);
  \coordinate(b) at (-2,-1);
  \coordinate(c) at (2,0);
  \coordinate(d) at (4,0);
  \draw (a) -- (b);
  \draw (c) -- (d);
  \draw (a) -- (-3,2);
  \draw (a) -- (-1,2);
  \draw (b) -- (-3,-2);
  \draw (b) -- (-1,-2);
  \draw[<->] (-1,0) -- (1,0);

  \draw (c) -- (1,1);
  \draw (c) -- (1,-1);
  \draw (d) -- (5,1);
  \draw (d) -- (5,-1);
  
  \draw[fill=white] (a) circle (6pt);
  \draw[fill=white] (b) circle (6pt);
  \draw[fill=white] (c) circle (6pt);
  \draw[fill=white] (d) circle (6pt);
 \end{tikzpicture} & 
 \begin{tikzpicture}[scale=0.5]
  \coordinate(a) at (-2,1);
  \coordinate(b) at (-2,-1);
  \coordinate(c) at (2,0);
  \coordinate(d) at (4,0);
  \draw (a) -- (b);
  \draw (c) -- (d);
  \draw (a) -- (-3,2);
  \draw (a) -- (-1,2);
  \draw (b) -- (-3,-2);
  \draw (b) -- (-1,-2);
  \draw[<->] (-1,0) -- (1,0);

  \draw (c) -- (1,1);
  \draw (c) -- (1,-1);
  \draw (d) -- (5,1);
  \draw (d) -- (5,-1);
  
  \draw[fill=black] (a) circle (6pt);
  \draw[fill=black] (b) circle (6pt);
  \draw[fill=black] (c) circle (6pt);
  \draw[fill=black] (d) circle (6pt);
 \end{tikzpicture}\\
 A white trivalent move (M1) & A black trivalent move (M3)

\end{tabular}
\setlength{\tabcolsep}{6pt}
\caption{\label{fig:moves_trivalent}Trivalent contraction-uncontraction moves}
\end{figure}

Given a reduced plabic graph $G$, one can associate to it a certain collection $\Fcal(G)\subset 2^{[n]}$ of \emph{face labels} of $G$. Namely, for each face $F$ of $G$, its \emph{label} $\l(F)\subset [n]$ contains all indices $j\in [n]$ such that $F$ is to the left of the strand $i\to j$. We set 
\[\Fcal(G)=\{\l(F)\mid F\text{ is a face of $G$}\}.\]
Clearly, if faces $F_1$ and $F_2$ share an edge then $\l(F_1)$ and $\l(F_2)$ have the same size, and thus $\Fcal(G)$ only contains sets of the same size.

\begin{theorem}[see \cite{OPS}]\label{thm:face_labels}
 For any reduced plabic graph $G$ with $\pi_G=\sigma\parr{k,n}$, the collection $\Fcal(G)$ is a maximal {by inclusion} weakly separated collection inside $[n]\choose k$, and conversely, $\WS\subset {[n]\choose k}$ is a maximal {by inclusion} weakly separated collection if and only if there is a reduced plabic graph $G$ with $\pi_G=\sigma\parr{k,n}$ and $\Fcal(G)=\WS$. 
\end{theorem}

Given that all reduced plabic graphs $G$ with $\pi_G=\sigma\parr{k,n}$ are connected by moves from Figure~\ref{fig:moves_plabic} which do not change the cardinality of $\Fcal(G)$, we get that all maximal \emph{by inclusion} weakly separated collections in $[n]\choose k$ have the same size, so the purity phenomenon (Theorem~\ref{thm:purity_known}, part~\ref{item:nchoosek}) follows from Theorems~\ref{thm:face_labels} and~\ref{thm:connected}. 

\begin{remark}
The first part of Theorem~\ref{thm:purity_known} follows from an analogous result in~\cite{OPS} concerning weakly separated collections \emph{inside a positroid}. Our construction, for the most part, can be repeated for positroids and Grassmann necklaces. However we prefer to just work with plabic graphs satisfying $\pi_G=\sigma\parr{k,n}$ for simplicity.
\end{remark}

It is quite straightforward to pass from a reduced plabic graph $G$ to its collection $\Fcal(G)$ of face labels. To go in the converse direction, we need the construction of \emph{plabic tilings} by Oh, Postnikov and Speyer~\cite{OPS} which we review in the next section. Figure~\ref{fig:strands} shows a plabic graph $G$ with $\pi_G=\sigma\parr{3,6}$ together with its face labels and its dual plabic tiling.

\subsection{Plabic tilings}\label{sect:plabic_tilings}
Our goal in this section is to assign a certain polygonal subdivision $\PTiling(\WS)$ of a convex polygon to each maximal {by inclusion} weakly separated collection $\WS\subset {[n]\choose k}$ of sets of the same size.

Fix a cyclic vector configuration $\Cyclic(n,3)=(\v_1,\dots,\v_n)$ in $\R^3$. For each subset $S\subset [n]$, define the point
\[\v_S:=\sum_{i\in S} \v_i.\]

Now, consider a maximal {by inclusion} chord-separated collection $\WS\subset {[n]\choose k}$ for some $1\leq k\leq n$. For each set $K\in {[n]\choose k-1}$, define the corresponding \emph{white clique} 
\[\WhCl(K):=\{S\in\WS\mid S\supset K\}.\]
Similarly, for each set $L\in {[n]\choose k+1}$, define the corresponding \emph{black clique}
\[\BlCl(L):=\{S\in\WS\mid S\subset L\}.\]
A clique is called \emph{non-trivial} if it contains at least three elements. Every white clique has a form $\{Ka_1,K a_2,\dots K a_r\}$ for some $a_1<a_2<\dots<a_r\in[n]$. Similarly, every black clique has a form $\{L- b_1,L- b_2,\dots L- b_s,\}$ for some $b_1<b_2<\dots<b_s\in[n]$. For a non-trivial white clique $\WhCl(K)$, its \emph{boundary} $\Bound\WhCl(K)$ is the following collection of pairs of subsets in $\WS$:
\[\Bound\WhCl(K)=\{(Ka_1,Ka_2),(Ka_2,Ka_3),\dots, (Ka_{r-1},Ka_r),(Ka_r,Ka_1)\}.\]
For a non-trivial black clique $\BlCl(L)$, its \emph{boundary} $\Bound\BlCl(L)$ is the following collection of pairs of subsets in $\WS$:
\[\Bound\BlCl(L)=\{(L- b_1,L- b_2),(L- b_2,L- b_3),\dots,(L- b_{r-1},L- b_r),(L- b_r,L- b_1)\}.\]

Oh-Postnikov-Speyer define the following two-dimensional complex $\PTiling(\WS)$ embedded in $\R^3$:
\begin{itemize}
 \item The vertices of $\PTiling(\WS)$ are $\{\v_S\mid S\in\WS\}$.
 \item The edges of $\PTiling(\WS)$ connect $\v_S$ and $\v_T$ for every $(S,T)$ belonging to the boundary of a non-trivial clique;
 \item The two-dimensional faces of $\PTiling(\WS)$ are polygons whose vertices form a non-trivial clique.
\end{itemize}

The plabic tiling corresponding to the collection from Figure~\ref{fig:strands} is depicted in Figure~\ref{fig:tiling} (top). The following result is proven in \cite{OPS}:
\begin{theorem}[\cite{OPS}]\label{thm:plabic_OPS}
If $\WS\subset {[n]\choose k}$ is a maximal \emph{by inclusion} chord separated collection of subsets, then it is also maximal \emph{by size}:
\[|\WS|=k(n-k)+1.\]
Moreover, $\PTiling(\WS)$ forms a polygonal subdivision of the convex polygon with $n$ vertices labeled by cyclic intervals in $[n]$ of length $k$. 
\end{theorem}

\begin{corollary}[\cite{OPS}]
 Given a reduced plabic graph $G$ with $\pi_G=\sigma\parr{k,n}$, the plabic tiling $\PTiling(\Fcal(G))$ is planar dual to $G$.
\end{corollary}

\section{Plabic graphs as sections of zonotopal tilings}\label{sect:sections}
Before we pass to the proof of Theorem~\ref{thm:purity_positive}, we would like to discuss informally how to view notions related to plabic graphs as sections of three-dimensional objects related to zonotopal tilings. 

Let $\Tiling$ be a fine zonotopal tiling of $\Zon(n,3)$. Let $\Tiling_i$ denote the horizontal sections of $\Tiling$ by hyperplanes $\Hyp_i:=\{(x,y,z)\mid z=i\}\subset \R^3$. First, as we have already mentioned in Theorem~\ref{thm:purity_positive}, the sections $\Tiling_i$ are triangulations of plabic tilings corresponding to maximal \emph{by inclusion} weakly separated collections inside $[n]\choose k$. Conversely, as we will see in the proof of Theorem~\ref{thm:purity_positive} in Section~\ref{sect:rk_3}, every triangulation of any plabic tiling is a horizontal section of some fine zonotopal tiling of $\Zon(n,3)$. Recall that the dual object to a triangulation of a plabic tiling is a trivalent reduced plabic graph $G$ such that $\pi_G=\sigma\parr{k,n}$. Thus we get \emph{plabic graphs satisfying $\pi_G=\sigma\parr{k,n}$ as duals of horizontal sections of fine zonotopal tilings $\Tiling$ of $\Zon(n,3)$. }

\def\PSP{{P}}

Next, the objects dual to fine zonotopal tilings of $\Zon(n,3)$ are \emph{three-dimensional pseudoplane arrangements}, which are analogous to the well-studied \emph{pseudoline arrangements} in two dimensions. Loosely speaking, given $\Tiling$ and any $j\in[n]$, we get a \emph{pseudoplane} $\PSP_j$ which is a smooth embedding of $\R^2$ (together with an orientation) into $\R^3$ that only intersects the tiles $\tau_X$ of $\Tiling$ satisfying $X_j=0$, i.e. $\PSP_j$ intersects all segments of $\Tiling$ that are parallel to $\v_j$.  In other words, the collection of vertex labels of $\Tiling$ lying on the \emph{negative} side of $\PSP_j$ consists precisely of those elements of $\Vert(\Tiling)$ that do not contain $j$. 

The formal definition of a pseudoplane arrangement is somewhat more complicated and involves other topological conditions like the intersection of $\PSP_i$ and $\PSP_j$ being homeomorphic to a line for $i\neq j\in [n]$. We will impose an extra condition that the normal to $\PSP_j$ at any point cannot be vertical for all $j\in [n]$.

\def\Rcal{ \mathcal{R}}
\def\Lcal{ \mathcal{L}}

The reason we are considering pseudoplanes is that their horizontal sections are strands in plabic graphs. Let $G$ be the plabic graph dual to $\Tiling_i$, and let $\ell=\pi_G^{-1}(j)\in [n]$ so that $G$ has a strand labeled $\ell\to j$. Then \emph{this strand $\ell\to j$ can be viewed as the intersection of $\PSP_j$ with $\Hyp_i$}. Let us explain this more rigorously. Let $\Rcal,\Lcal\subset {[n]\choose k}$ be the collections of sets in $\Vert(\Tiling_i)$ lying on the right, resp., left side of the strand $\ell\to j$. So 
\[\Rcal=\{S\in\Vert(\Tiling_i)\mid j\not\in S\};\quad \Lcal=\{T\in\Vert(\Tiling)\mid j\in T\}.\]
From this one easily observes that $\Rcal$, resp., $\Lcal$ are the collections of sets in $\Vert(\Tiling_i)$ that are on the negative, resp., positive sides of $\PSP_j$. Finally, by our assumption on the normal of $\PSP_j$ not being vertical, it follows that $\PSP_j\cap \Zon(n,3)\cap \Hyp_i$ is a simple curve that divides the polygon $\Zon(n,3)\cap\Hyp_i$ into two parts in essentially the same way as the strand $\ell\to j$.

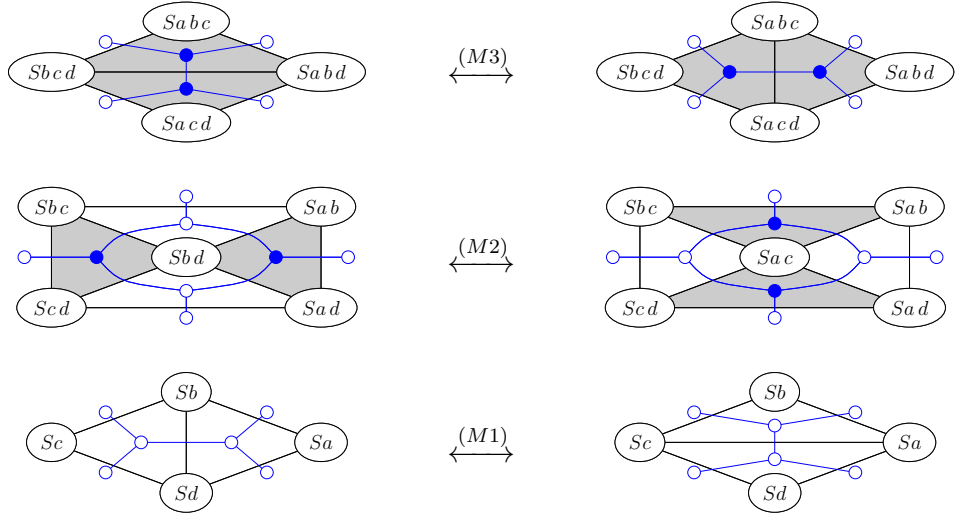
\begin{figure}
 \setlength{\tabcolsep}{12pt}

\begin{tabular}{ccc}
\scalebox{0.8}{
\begin{tikzpicture}[every node/.style={scale=0.8}]
\node[draw,ellipse,black,fill=white] (node1) at (2.24,2.24) {$Sa$};
\node[draw,ellipse,black,fill=white] (node2) at (0.00,3.08) {$Sb$};
\node[draw,ellipse,black,fill=white] (node3) at (-2.24,2.24) {$Sc$};
\node[draw,ellipse,black,fill=white] (node4) at (-0.00,1.40) {$Sd$};
\node[draw,ellipse,black,fill=white] (node12) at (2.24,6.16) {$Sa\,b$};
\node[draw,ellipse,black,fill=white] (node14) at (2.24,4.48) {$Sa\,d$};
\node[draw,ellipse,black,fill=white] (node23) at (-2.24,6.16) {$Sb\,c$};
\node[draw,ellipse,black,fill=white] (node24) at (-0.00,5.32) {$Sb\,d$};
\node[draw,ellipse,black,fill=white] (node34) at (-2.24,4.48) {$Sc\,d$};
\node[draw,ellipse,black,fill=white] (node123) at (0.00,9.24) {$Sa\,b\,c$};
\node[draw,ellipse,black,fill=white] (node124) at (2.24,8.40) {$Sa\,b\,d$};
\node[draw,ellipse,black,fill=white] (node134) at (-0.00,7.56) {$Sa\,c\,d$};
\node[draw,ellipse,black,fill=white] (node234) at (-2.24,8.40) {$Sb\,c\,d$};
\coordinate (wnode0n1) at (0.75,2.24);
\coordinate (wnode0n2) at (-0.75,2.24);
\fill [opacity=0.2,black] (node12.center) -- (node14.center) -- (node24.center) -- cycle;
\coordinate (bnode124n1) at (1.49,5.32);
\coordinate (wnode2n1) at (-0.00,5.88);
\coordinate (wnode4n1) at (-0.00,4.76);
\fill [opacity=0.2,black] (node23.center) -- (node24.center) -- (node34.center) -- cycle;
\coordinate (bnode234n1) at (-1.49,5.32);
\fill [opacity=0.2,black] (node123.center) -- (node124.center) -- (node234.center) -- cycle;
\coordinate (bnode1234n1) at (-0.00,8.68);
\fill [opacity=0.2,black] (node124.center) -- (node134.center) -- (node234.center) -- cycle;
\coordinate (bnode1234n2) at (-0.00,8.12);
\draw[line width=0.03mm,black] (node1) -- (node2);
\draw[line width=0.03mm,black] (node1) -- (node4);
\draw[line width=0.03mm,black] (node2) -- (node1);
\draw[line width=0.03mm,black] (node2) -- (node3);
\draw[line width=0.03mm,black] (node2) -- (node4);
\draw[line width=0.03mm,black] (node3) -- (node2);
\draw[line width=0.03mm,black] (node3) -- (node4);
\draw[line width=0.03mm,black] (node4) -- (node1);
\draw[line width=0.03mm,black] (node4) -- (node2);
\draw[line width=0.03mm,black] (node4) -- (node3);
\draw[line width=0.03mm,black] (node12) -- (node14);
\draw[line width=0.03mm,black] (node12) -- (node23);
\draw[line width=0.03mm,black] (node12) -- (node24);
\draw[line width=0.03mm,black] (node14) -- (node12);
\draw[line width=0.03mm,black] (node14) -- (node24);
\draw[line width=0.03mm,black] (node14) -- (node34);
\draw[line width=0.03mm,black] (node23) -- (node12);
\draw[line width=0.03mm,black] (node23) -- (node24);
\draw[line width=0.03mm,black] (node23) -- (node34);
\draw[line width=0.03mm,black] (node24) -- (node12);
\draw[line width=0.03mm,black] (node24) -- (node14);
\draw[line width=0.03mm,black] (node24) -- (node23);
\draw[line width=0.03mm,black] (node24) -- (node34);
\draw[line width=0.03mm,black] (node34) -- (node14);
\draw[line width=0.03mm,black] (node34) -- (node23);
\draw[line width=0.03mm,black] (node34) -- (node24);
\draw[line width=0.03mm,black] (node123) -- (node124);
\draw[line width=0.03mm,black] (node123) -- (node234);
\draw[line width=0.03mm,black] (node124) -- (node123);
\draw[line width=0.03mm,black] (node124) -- (node134);
\draw[line width=0.03mm,black] (node124) -- (node234);
\draw[line width=0.03mm,black] (node134) -- (node124);
\draw[line width=0.03mm,black] (node134) -- (node234);
\draw[line width=0.03mm,black] (node234) -- (node123);
\draw[line width=0.03mm,black] (node234) -- (node124);
\draw[line width=0.03mm,black] (node234) -- (node134);
\coordinate (node0) at (-0.00,2.24);
\coordinate (node12) at (1.34,2.74);
\coordinate (node14) at (1.34,1.74);
\coordinate (node23) at (-1.34,2.74);
\coordinate (node34) at (-1.34,1.74);
% \draw[->,color={rgb:red,170;green,85;blue,162},line width=1.5pt] (0.96,2.77).. controls (1.12,2.66) .. (1.12,2.24);
% \draw[->,color={rgb:red,170;green,85;blue,162},line width=1.5pt] (1.12,2.24).. controls (1.12,1.82) .. (0.96,1.71);
\draw[blue] (node12) -- (wnode0n1);
\draw[blue] (node23) -- (wnode0n2);
\draw[blue] (node34) -- (wnode0n2);
\draw[blue] (node14) -- (wnode0n1);
\draw[blue] (wnode0n1) -- (wnode0n2);
\draw[blue,fill=white] (wnode0n1) circle (3pt);
\draw[blue,fill=white] (wnode0n2) circle (3pt);
\draw[blue,fill=white] (node12) circle (3pt);
\draw[blue,fill=white] (node14) circle (3pt);
\draw[blue,fill=white] (node23) circle (3pt);
\draw[blue,fill=white] (node34) circle (3pt);
\coordinate (node1) at (2.69,5.32);
\coordinate (node2) at (-0.00,5.88);
\coordinate (node3) at (-2.69,5.32);
\coordinate (node4) at (-0.00,4.76);
\coordinate (node123) at (0.00,6.33);
\coordinate (node124) at (1.49,5.32);
\coordinate (node134) at (-0.00,4.31);
\coordinate (node234) at (-1.49,5.32);
\draw[blue] (2.69,5.32).. controls (2.24,5.32) .. (1.49,5.32);
\draw[blue] (-0.00,5.88).. controls (0.00,6.16) .. (0.00,6.33);
\draw[blue] (-0.00,5.88).. controls (-1.12,5.74) .. (-1.49,5.32);
\draw[blue] (-0.00,5.88).. controls (1.12,5.74) .. (1.49,5.32);
\draw[blue] (-2.69,5.32).. controls (-2.24,5.32) .. (-1.49,5.32);
\draw[blue] (-0.00,4.76).. controls (1.12,4.90) .. (1.49,5.32);
\draw[blue] (-0.00,4.76).. controls (-1.12,4.90) .. (-1.49,5.32);
\draw[blue] (-0.00,4.76).. controls (-0.00,4.48) .. (-0.00,4.31);
\draw[blue] (0.00,6.33).. controls (0.00,6.16) .. (-0.00,5.88);
\draw[blue] (1.49,5.32).. controls (2.24,5.32) .. (2.69,5.32);
\draw[blue] (1.49,5.32).. controls (1.12,4.90) .. (-0.00,4.76);
\draw[blue] (1.49,5.32).. controls (1.12,5.74) .. (-0.00,5.88);
\draw[blue] (-0.00,4.31).. controls (-0.00,4.48) .. (-0.00,4.76);
\draw[blue] (-1.49,5.32).. controls (-1.12,5.74) .. (-0.00,5.88);
\draw[blue] (-1.49,5.32).. controls (-1.12,4.90) .. (-0.00,4.76);
\draw[blue] (-1.49,5.32).. controls (-2.24,5.32) .. (-2.69,5.32);
% \draw[->,color={rgb:red,170;green,85;blue,162},line width=1.5pt] (-0.45,6.21).. controls (0.00,6.16) .. (0.56,5.95);
% \draw[->,color={rgb:red,170;green,85;blue,162},line width=1.5pt] (0.56,5.95).. controls (1.12,5.74) .. (1.12,5.32);
% \draw[->,color={rgb:red,170;green,85;blue,162},line width=1.5pt] (1.12,5.32).. controls (1.12,4.90) .. (0.56,4.69);
% \draw[->,color={rgb:red,170;green,85;blue,162},line width=1.5pt] (0.56,4.69).. controls (-0.00,4.48) .. (-0.45,4.43);
\draw[blue,fill=white] (node1) circle (3pt);
\draw[blue,fill=white] (node2) circle (3pt);
\draw[blue,fill=white] (node3) circle (3pt);
\draw[blue,fill=white] (node4) circle (3pt);
\draw[blue,fill=white] (node123) circle (3pt);
\draw[blue,fill=blue] (node124) circle (3pt);
\draw[blue,fill=white] (node134) circle (3pt);
\draw[blue,fill=blue] (node234) circle (3pt);
\coordinate (node12) at (1.34,8.90);
\coordinate (node14) at (1.34,7.90);
\coordinate (node23) at (-1.34,8.90);
\coordinate (node34) at (-1.34,7.90);
\coordinate (node1234) at (-0.00,8.40);
% \draw[->,color={rgb:red,170;green,85;blue,162},line width=1.5pt] (-0.96,8.93).. controls (-1.12,8.82) .. (-1.12,8.40);
% \draw[->,color={rgb:red,170;green,85;blue,162},line width=1.5pt] (-1.12,8.40).. controls (-1.12,7.98) .. (-0.96,7.87);
\draw[blue] (node12) -- (bnode1234n1);
\draw[blue] (node23) -- (bnode1234n1);
\draw[blue] (node34) -- (bnode1234n2);
\draw[blue] (node14) -- (bnode1234n2);
\draw[blue] (bnode1234n1) -- (bnode1234n2);
\draw[blue,fill=blue] (bnode1234n1) circle (3pt);
\draw[blue,fill=blue] (bnode1234n2) circle (3pt);
\draw[blue,fill=white] (node12) circle (3pt);
\draw[blue,fill=white] (node14) circle (3pt);
\draw[blue,fill=white] (node23) circle (3pt);
\draw[blue,fill=white] (node34) circle (3pt);
\node[draw,ellipse,black,fill=white] (node1) at (2.24,2.24) {$Sa$};
\node[draw,ellipse,black,fill=white] (node2) at (0.00,3.08) {$Sb$};
\node[draw,ellipse,black,fill=white] (node3) at (-2.24,2.24) {$Sc$};
\node[draw,ellipse,black,fill=white] (node4) at (-0.00,1.40) {$Sd$};
\node[draw,ellipse,black,fill=white] (node12) at (2.24,6.16) {$Sa\,b$};
\node[draw,ellipse,black,fill=white] (node14) at (2.24,4.48) {$Sa\,d$};
\node[draw,ellipse,black,fill=white] (node23) at (-2.24,6.16) {$Sb\,c$};
\node[draw,ellipse,black,fill=white] (node24) at (-0.00,5.32) {$Sb\,d$};
\node[draw,ellipse,black,fill=white] (node34) at (-2.24,4.48) {$Sc\,d$};
\node[draw,ellipse,black,fill=white] (node123) at (0.00,9.24) {$Sa\,b\,c$};
\node[draw,ellipse,black,fill=white] (node124) at (2.24,8.40) {$Sa\,b\,d$};
\node[draw,ellipse,black,fill=white] (node134) at (-0.00,7.56) {$Sa\,c\,d$};
\node[draw,ellipse,black,fill=white] (node234) at (-2.24,8.40) {$Sb\,c\,d$};
\end{tikzpicture}}
&
\scalebox{0.8}{
\begin{tikzpicture}[every node/.style={scale=1.2}]
\draw[color=white] (0,0) circle (3pt);
\node[] (arrow1) at (0.00,1.05) {$\xleftrightarrow{(M1)}$};
\node[] (arrow2) at (0.00,4.20) {$\xleftrightarrow{(M2)}$};
\node[] (arrow3) at (0.00,7.35) {$\xleftrightarrow{(M3)}$};
\end{tikzpicture}}
&
\scalebox{0.8}{
\begin{tikzpicture}[every node/.style={scale=0.8}]
\node[draw,ellipse,black,fill=white] (node1) at (2.24,2.24) {$Sa$};
\node[draw,ellipse,black,fill=white] (node2) at (0.00,3.08) {$Sb$};
\node[draw,ellipse,black,fill=white] (node3) at (-2.24,2.24) {$Sc$};
\node[draw,ellipse,black,fill=white] (node4) at (-0.00,1.40) {$Sd$};
\node[draw,ellipse,black,fill=white] (node12) at (2.24,6.16) {$Sa\,b$};
\node[draw,ellipse,black,fill=white] (node13) at (0.00,5.32) {$Sa\,c$};
\node[draw,ellipse,black,fill=white] (node14) at (2.24,4.48) {$Sa\,d$};
\node[draw,ellipse,black,fill=white] (node23) at (-2.24,6.16) {$Sb\,c$};
\node[draw,ellipse,black,fill=white] (node34) at (-2.24,4.48) {$Sc\,d$};
\node[draw,ellipse,black,fill=white] (node123) at (0.00,9.24) {$Sa\,b\,c$};
\node[draw,ellipse,black,fill=white] (node124) at (2.24,8.40) {$Sa\,b\,d$};
\node[draw,ellipse,black,fill=white] (node134) at (-0.00,7.56) {$Sa\,c\,d$};
\node[draw,ellipse,black,fill=white] (node234) at (-2.24,8.40) {$Sb\,c\,d$};
\coordinate (wnode0n1) at (0.00,2.52);
\coordinate (wnode0n2) at (-0.00,1.96);
\coordinate (wnode1n1) at (1.49,5.32);
\fill [opacity=0.2,black] (node12.center) -- (node13.center) -- (node23.center) -- cycle;
\coordinate (bnode123n1) at (0.00,5.88);
\fill [opacity=0.2,black] (node13.center) -- (node14.center) -- (node34.center) -- cycle;
\coordinate (bnode134n1) at (-0.00,4.76);
\coordinate (wnode3n1) at (-1.49,5.32);
\fill [opacity=0.2,black] (node123.center) -- (node124.center) -- (node134.center) -- cycle;
\coordinate (bnode1234n1) at (0.75,8.40);
\fill [opacity=0.2,black] (node123.center) -- (node134.center) -- (node234.center) -- cycle;
\coordinate (bnode1234n2) at (-0.75,8.40);
\draw[line width=0.03mm,black] (node1) -- (node2);
\draw[line width=0.03mm,black] (node1) -- (node3);
\draw[line width=0.03mm,black] (node1) -- (node4);
\draw[line width=0.03mm,black] (node2) -- (node1);
\draw[line width=0.03mm,black] (node2) -- (node3);
\draw[line width=0.03mm,black] (node3) -- (node1);
\draw[line width=0.03mm,black] (node3) -- (node2);
\draw[line width=0.03mm,black] (node3) -- (node4);
\draw[line width=0.03mm,black] (node4) -- (node1);
\draw[line width=0.03mm,black] (node4) -- (node3);
\draw[line width=0.03mm,black] (node12) -- (node13);
\draw[line width=0.03mm,black] (node12) -- (node14);
\draw[line width=0.03mm,black] (node12) -- (node23);
\draw[line width=0.03mm,black] (node13) -- (node12);
\draw[line width=0.03mm,black] (node13) -- (node14);
\draw[line width=0.03mm,black] (node13) -- (node23);
\draw[line width=0.03mm,black] (node13) -- (node34);
\draw[line width=0.03mm,black] (node14) -- (node12);
\draw[line width=0.03mm,black] (node14) -- (node13);
\draw[line width=0.03mm,black] (node14) -- (node34);
\draw[line width=0.03mm,black] (node23) -- (node12);
\draw[line width=0.03mm,black] (node23) -- (node13);
\draw[line width=0.03mm,black] (node23) -- (node34);
\draw[line width=0.03mm,black] (node34) -- (node13);
\draw[line width=0.03mm,black] (node34) -- (node14);
\draw[line width=0.03mm,black] (node34) -- (node23);
\draw[line width=0.03mm,black] (node123) -- (node124);
\draw[line width=0.03mm,black] (node123) -- (node134);
\draw[line width=0.03mm,black] (node123) -- (node234);
\draw[line width=0.03mm,black] (node124) -- (node123);
\draw[line width=0.03mm,black] (node124) -- (node134);
\draw[line width=0.03mm,black] (node134) -- (node123);
\draw[line width=0.03mm,black] (node134) -- (node124);
\draw[line width=0.03mm,black] (node134) -- (node234);
\draw[line width=0.03mm,black] (node234) -- (node123);
\draw[line width=0.03mm,black] (node234) -- (node134);
\coordinate (node0) at (-0.00,2.24);
\coordinate (node12) at (1.34,2.74);
\coordinate (node14) at (1.34,1.74);
\coordinate (node23) at (-1.34,2.74);
\coordinate (node34) at (-1.34,1.74);
% \draw[->,color={rgb:red,170;green,85;blue,162},line width=1.5pt] (0.96,2.77).. controls (1.12,2.66) .. (1.12,2.24);
% \draw[->,color={rgb:red,170;green,85;blue,162},line width=1.5pt] (1.12,2.24).. controls (1.12,1.82) .. (0.96,1.71);
\draw[blue] (node12) -- (wnode0n1);
\draw[blue] (node23) -- (wnode0n1);
\draw[blue] (node34) -- (wnode0n2);
\draw[blue] (node14) -- (wnode0n2);
\draw[blue] (wnode0n1) -- (wnode0n2);
\draw[blue,fill=white] (wnode0n1) circle (3pt);
\draw[blue,fill=white] (wnode0n2) circle (3pt);
\draw[blue,fill=white] (node12) circle (3pt);
\draw[blue,fill=white] (node14) circle (3pt);
\draw[blue,fill=white] (node23) circle (3pt);
\draw[blue,fill=white] (node34) circle (3pt);
\coordinate (node1) at (1.49,5.32);
\coordinate (node2) at (0.00,6.33);
\coordinate (node3) at (-1.49,5.32);
\coordinate (node4) at (-0.00,4.31);
\coordinate (node123) at (0.00,5.88);
\coordinate (node124) at (2.69,5.32);
\coordinate (node134) at (-0.00,4.76);
\coordinate (node234) at (-2.69,5.32);
\draw[blue] (1.49,5.32).. controls (1.12,5.74) .. (0.00,5.88);
\draw[blue] (1.49,5.32).. controls (1.12,4.90) .. (-0.00,4.76);
\draw[blue] (1.49,5.32).. controls (2.24,5.32) .. (2.69,5.32);
\draw[blue] (0.00,6.33).. controls (0.00,6.16) .. (0.00,5.88);
\draw[blue] (-1.49,5.32).. controls (-1.12,5.74) .. (0.00,5.88);
\draw[blue] (-1.49,5.32).. controls (-2.24,5.32) .. (-2.69,5.32);
\draw[blue] (-1.49,5.32).. controls (-1.12,4.90) .. (-0.00,4.76);
\draw[blue] (-0.00,4.31).. controls (-0.00,4.48) .. (-0.00,4.76);
\draw[blue] (0.00,5.88).. controls (1.12,5.74) .. (1.49,5.32);
\draw[blue] (0.00,5.88).. controls (-1.12,5.74) .. (-1.49,5.32);
\draw[blue] (0.00,5.88).. controls (0.00,6.16) .. (0.00,6.33);
\draw[blue] (2.69,5.32).. controls (2.24,5.32) .. (1.49,5.32);
\draw[blue] (-0.00,4.76).. controls (1.12,4.90) .. (1.49,5.32);
\draw[blue] (-0.00,4.76).. controls (-0.00,4.48) .. (-0.00,4.31);
\draw[blue] (-0.00,4.76).. controls (-1.12,4.90) .. (-1.49,5.32);
\draw[blue] (-2.69,5.32).. controls (-2.24,5.32) .. (-1.49,5.32);
% \draw[->,color={rgb:red,170;green,85;blue,162},line width=1.5pt] (0.45,6.21).. controls (0.00,6.16) .. (-0.56,5.95);
% \draw[->,color={rgb:red,170;green,85;blue,162},line width=1.5pt] (-0.56,5.95).. controls (-1.12,5.74) .. (-1.12,5.32);
% \draw[->,color={rgb:red,170;green,85;blue,162},line width=1.5pt] (-1.12,5.32).. controls (-1.12,4.90) .. (-0.56,4.69);
% \draw[->,color={rgb:red,170;green,85;blue,162},line width=1.5pt] (-0.56,4.69).. controls (-0.00,4.48) .. (0.45,4.43);
\draw[blue,fill=white] (node1) circle (3pt);
\draw[blue,fill=white] (node2) circle (3pt);
\draw[blue,fill=white] (node3) circle (3pt);
\draw[blue,fill=white] (node4) circle (3pt);
\draw[blue,fill=blue] (node123) circle (3pt);
\draw[blue,fill=white] (node124) circle (3pt);
\draw[blue,fill=blue] (node134) circle (3pt);
\draw[blue,fill=white] (node234) circle (3pt);
\coordinate (node12) at (1.34,8.90);
\coordinate (node14) at (1.34,7.90);
\coordinate (node23) at (-1.34,8.90);
\coordinate (node34) at (-1.34,7.90);
\coordinate (node1234) at (-0.00,8.40);
% \draw[->,color={rgb:red,170;green,85;blue,162},line width=1.5pt] (-0.96,8.93).. controls (-1.12,8.82) .. (-1.12,8.40);
% \draw[->,color={rgb:red,170;green,85;blue,162},line width=1.5pt] (-1.12,8.40).. controls (-1.12,7.98) .. (-0.96,7.87);
\draw[blue] (node12) -- (bnode1234n1);
\draw[blue] (node23) -- (bnode1234n2);
\draw[blue] (node34) -- (bnode1234n2);
\draw[blue] (node14) -- (bnode1234n1);
\draw[blue] (bnode1234n1) -- (bnode1234n2);
\draw[blue,fill=blue] (bnode1234n1) circle (3pt);
\draw[blue,fill=blue] (bnode1234n2) circle (3pt);
\draw[blue,fill=white] (node12) circle (3pt);
\draw[blue,fill=white] (node14) circle (3pt);
\draw[blue,fill=white] (node23) circle (3pt);
\draw[blue,fill=white] (node34) circle (3pt);
\node[draw,ellipse,black,fill=white] (node1) at (2.24,2.24) {$Sa$};
\node[draw,ellipse,black,fill=white] (node2) at (0.00,3.08) {$Sb$};
\node[draw,ellipse,black,fill=white] (node3) at (-2.24,2.24) {$Sc$};
\node[draw,ellipse,black,fill=white] (node4) at (-0.00,1.40) {$Sd$};
\node[draw,ellipse,black,fill=white] (node12) at (2.24,6.16) {$Sa\,b$};
\node[draw,ellipse,black,fill=white] (node13) at (0.00,5.32) {$Sa\,c$};
\node[draw,ellipse,black,fill=white] (node14) at (2.24,4.48) {$Sa\,d$};
\node[draw,ellipse,black,fill=white] (node23) at (-2.24,6.16) {$Sb\,c$};
\node[draw,ellipse,black,fill=white] (node34) at (-2.24,4.48) {$Sc\,d$};
\node[draw,ellipse,black,fill=white] (node123) at (0.00,9.24) {$Sa\,b\,c$};
\node[draw,ellipse,black,fill=white] (node124) at (2.24,8.40) {$Sa\,b\,d$};
\node[draw,ellipse,black,fill=white] (node134) at (-0.00,7.56) {$Sa\,c\,d$};
\node[draw,ellipse,black,fill=white] (node234) at (-2.24,8.40) {$Sb\,c\,d$};
\end{tikzpicture}}
\\

\end{tabular}
\setlength{\tabcolsep}{6pt}

 \caption{\label{fig:mutation} A mutation $\Tiling\leftrightarrow\Tiling'$ corresponds to performing three simultaneous moves on plabic graphs: (M1) on level $|S|+1$, (M2) on level $|S|+2$, and (M3) on level $|S|+3$. 
 %The sections of the pseudoplane $\PSP_a$ are shown in purple.
 }
\end{figure}

Finally we consider the well-studied notion of \emph{mutations} of zonotopal tilings. It is easy to see that there are only two fine zonotopal tilings of $\Zon(4,3)$. Given any two distinct fine zonotopal tilings $\Tiling$ and $\Tiling'$ of $\Zon(n,3)$, we say that they \emph{differ by a mutation} if there is a set $S\subset [n]$ and four indices $a,b,c,d\in [n]-S$ such that for any tile $\tau_X$ we have 
\[X\in\Tiling\Longleftrightarrow X\in\Tiling'\]
unless $S\subset X^+\subset Sabcd$ and $X^0\subset \{a,b,c,d\}$. This can be reformulated as follows. Consider a signed set $Y=(S,[n]-Sabcd)$. Thus $\tau_Y$ is a zonotope that is combinatorially equivalent to $\Zon(4,3)$. Hence there are only two fine zonotopal tilings of $\tau_Y$  which we denote $\Tiling_0$ and $\Tiling_0'$ (the tilings $\Tiling_0$ and $\Tiling_0'$ are shown in Figure~\ref{fig:mutation}). Then $\Tiling$ and $\Tiling'$ differ by a mutation if and only if they coincide on the complement of $\tau_Y$ and their restrictions to $\tau_Y$ are $\Tiling_0$ and $\Tiling_0'$ respectively. 

As Figure~\ref{fig:mutation} suggests, \emph{performing a mutation on $\Tiling$ is equivalent to performing moves (M1), (M2), and (M3) on the corresponding plabic graphs}. 

\begin{remark}
 For plabic graphs $G$ with $\pi_G=\sigma\parr{k,n}$, our construction allows to deduce Postnikov's Theorem~\ref{thm:connected} from Ziegler's~\cite[Theorem 4.1(G)]{Ziegler} where he shows that the higher Bruhat order $B(n,k)$ is a graded poset with a unique minimal and maximal elements. Indeed, fine zonotopal tilings of $\Zon(n,3)$ correspond to elements of $B(n,2)$, and their mutations correspond to covering relations in $B(n,2)$. 
\end{remark}

\section{Proof of Theorem~\ref{thm:purity_positive}}\label{sect:rk_3}

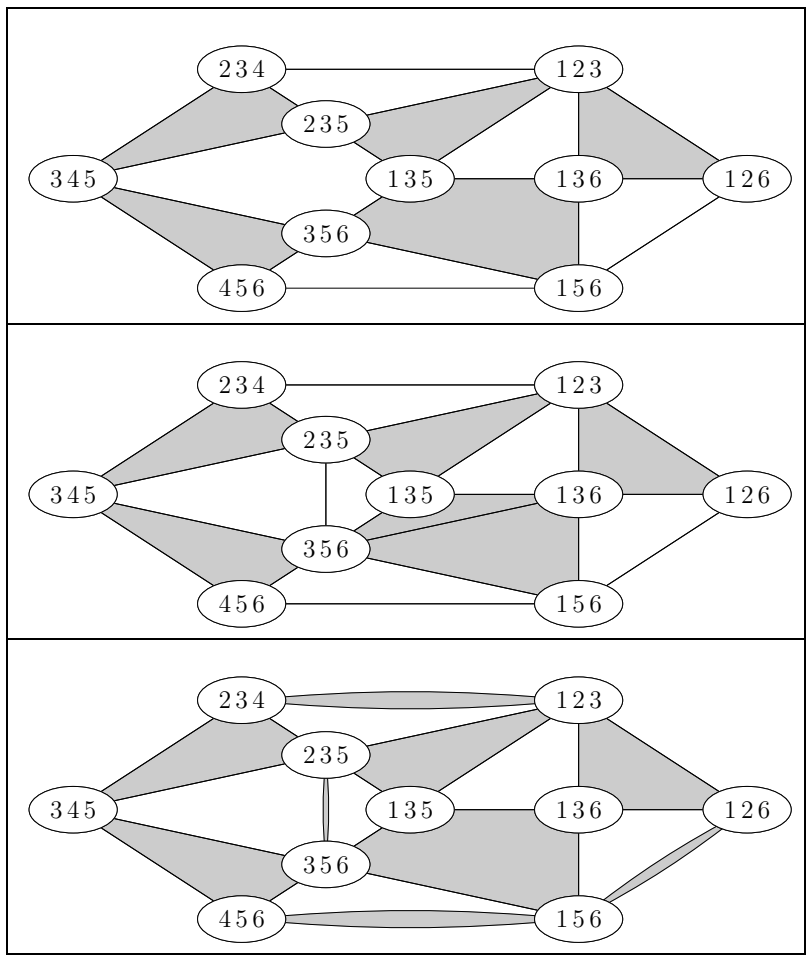
\begin{figure}
\begin{tabular}{|c|}\hline
\\
 \scalebox{0.8}{
\begin{tikzpicture}
\node[draw,ellipse,black,fill=white] (node123) at (2.80,5.74) {$1\,2\,3$};
\node[draw,ellipse,black,fill=white] (node126) at (5.60,3.92) {$1\,2\,6$};
\node[draw,ellipse,black,fill=white] (node135) at (-0.00,3.92) {$1\,3\,5$};
\node[draw,ellipse,black,fill=white] (node136) at (2.80,3.92) {$1\,3\,6$};
\node[draw,ellipse,black,fill=white] (node156) at (2.80,2.10) {$1\,5\,6$};
\node[draw,ellipse,black,fill=white] (node234) at (-2.80,5.74) {$2\,3\,4$};
\node[draw,ellipse,black,fill=white] (node235) at (-1.40,4.83) {$2\,3\,5$};
\node[draw,ellipse,black,fill=white] (node345) at (-5.60,3.92) {$3\,4\,5$};
\node[draw,ellipse,black,fill=white] (node356) at (-1.40,3.01) {$3\,5\,6$};
\node[draw,ellipse,black,fill=white] (node456) at (-2.80,2.10) {$4\,5\,6$};
\fill [opacity=0.2,black] (node123.center) -- (node126.center) -- (node136.center) -- cycle;
\fill [opacity=0.2,black] (node123.center) -- (node135.center) -- (node235.center) -- cycle;
\fill [opacity=0.2,black] (node135.center) -- (node136.center) -- (node156.center) -- (node356.center) -- cycle;
% \fill [opacity=0.2,black] (node136.center) -- (node156.center) -- (node356.center) -- cycle;
\fill [opacity=0.2,black] (node234.center) -- (node235.center) -- (node345.center) -- cycle;
\fill [opacity=0.2,black] (node345.center) -- (node356.center) -- (node456.center) -- cycle;
\draw[line width=0.04mm,black] (node123) -- (node126);
\draw[line width=0.04mm,black] (node123) -- (node135);
\draw[line width=0.04mm,black] (node123) -- (node136);
\draw[line width=0.04mm,black] (node123) -- (node234);
\draw[line width=0.04mm,black] (node123) -- (node235);
\draw[line width=0.04mm,black] (node126) -- (node123);
\draw[line width=0.04mm,black] (node126) -- (node136);
\draw[line width=0.04mm,black] (node126) -- (node156);
\draw[line width=0.04mm,black] (node135) -- (node123);
\draw[line width=0.04mm,black] (node135) -- (node136);
\draw[line width=0.04mm,black] (node135) -- (node235);
\draw[line width=0.04mm,black] (node135) -- (node356);
\draw[line width=0.04mm,black] (node136) -- (node123);
\draw[line width=0.04mm,black] (node136) -- (node126);
\draw[line width=0.04mm,black] (node136) -- (node135);
\draw[line width=0.04mm,black] (node136) -- (node156);
\draw[line width=0.04mm,black] (node156) -- (node126);
\draw[line width=0.04mm,black] (node156) -- (node136);
\draw[line width=0.04mm,black] (node156) -- (node356);
\draw[line width=0.04mm,black] (node156) -- (node456);
\draw[line width=0.04mm,black] (node234) -- (node123);
\draw[line width=0.04mm,black] (node234) -- (node235);
\draw[line width=0.04mm,black] (node234) -- (node345);
\draw[line width=0.04mm,black] (node235) -- (node123);
\draw[line width=0.04mm,black] (node235) -- (node135);
\draw[line width=0.04mm,black] (node235) -- (node234);
\draw[line width=0.04mm,black] (node235) -- (node345);
\draw[line width=0.04mm,black] (node345) -- (node234);
\draw[line width=0.04mm,black] (node345) -- (node235);
\draw[line width=0.04mm,black] (node345) -- (node356);
\draw[line width=0.04mm,black] (node345) -- (node456);
\draw[line width=0.04mm,black] (node356) -- (node135);
\draw[line width=0.04mm,black] (node356) -- (node156);
\draw[line width=0.04mm,black] (node356) -- (node345);
\draw[line width=0.04mm,black] (node356) -- (node456);
\draw[line width=0.04mm,black] (node456) -- (node345);
\draw[line width=0.04mm,black] (node456) -- (node356);
\node[draw,ellipse,black,fill=white] (node123) at (2.80,5.74) {$1\,2\,3$};
\node[draw,ellipse,black,fill=white] (node126) at (5.60,3.92) {$1\,2\,6$};
\node[draw,ellipse,black,fill=white] (node135) at (-0.00,3.92) {$1\,3\,5$};
\node[draw,ellipse,black,fill=white] (node136) at (2.80,3.92) {$1\,3\,6$};
\node[draw,ellipse,black,fill=white] (node156) at (2.80,2.10) {$1\,5\,6$};
\node[draw,ellipse,black,fill=white] (node234) at (-2.80,5.74) {$2\,3\,4$};
\node[draw,ellipse,black,fill=white] (node235) at (-1.40,4.83) {$2\,3\,5$};
\node[draw,ellipse,black,fill=white] (node345) at (-5.60,3.92) {$3\,4\,5$};
\node[draw,ellipse,black,fill=white] (node356) at (-1.40,3.01) {$3\,5\,6$};
\node[draw,ellipse,black,fill=white] (node456) at (-2.80,2.10) {$4\,5\,6$};
\end{tikzpicture}}
\\\hline\\
\scalebox{0.8}{
\begin{tikzpicture}
\node[draw,ellipse,black,fill=white] (node123) at (2.80,5.74) {$1\,2\,3$};
\node[draw,ellipse,black,fill=white] (node126) at (5.60,3.92) {$1\,2\,6$};
\node[draw,ellipse,black,fill=white] (node135) at (-0.00,3.92) {$1\,3\,5$};
\node[draw,ellipse,black,fill=white] (node136) at (2.80,3.92) {$1\,3\,6$};
\node[draw,ellipse,black,fill=white] (node156) at (2.80,2.10) {$1\,5\,6$};
\node[draw,ellipse,black,fill=white] (node234) at (-2.80,5.74) {$2\,3\,4$};
\node[draw,ellipse,black,fill=white] (node235) at (-1.40,4.83) {$2\,3\,5$};
\node[draw,ellipse,black,fill=white] (node345) at (-5.60,3.92) {$3\,4\,5$};
\node[draw,ellipse,black,fill=white] (node356) at (-1.40,3.01) {$3\,5\,6$};
\node[draw,ellipse,black,fill=white] (node456) at (-2.80,2.10) {$4\,5\,6$};
\fill [opacity=0.2,black] (node123.center) -- (node126.center) -- (node136.center) -- cycle;
\fill [opacity=0.2,black] (node123.center) -- (node135.center) -- (node235.center) -- cycle;
\fill [opacity=0.2,black] (node135.center) -- (node136.center) -- (node356.center) -- cycle;
\fill [opacity=0.2,black] (node136.center) -- (node156.center) -- (node356.center) -- cycle;
\fill [opacity=0.2,black] (node234.center) -- (node235.center) -- (node345.center) -- cycle;
\fill [opacity=0.2,black] (node345.center) -- (node356.center) -- (node456.center) -- cycle;
\draw[line width=0.04mm,black] (node123) -- (node126);
\draw[line width=0.04mm,black] (node123) -- (node135);
\draw[line width=0.04mm,black] (node123) -- (node136);
\draw[line width=0.04mm,black] (node123) -- (node234);
\draw[line width=0.04mm,black] (node123) -- (node235);
\draw[line width=0.04mm,black] (node126) -- (node123);
\draw[line width=0.04mm,black] (node126) -- (node136);
\draw[line width=0.04mm,black] (node126) -- (node156);
\draw[line width=0.04mm,black] (node135) -- (node123);
\draw[line width=0.04mm,black] (node135) -- (node136);
\draw[line width=0.04mm,black] (node135) -- (node235);
\draw[line width=0.04mm,black] (node135) -- (node356);
\draw[line width=0.04mm,black] (node136) -- (node123);
\draw[line width=0.04mm,black] (node136) -- (node126);
\draw[line width=0.04mm,black] (node136) -- (node135);
\draw[line width=0.04mm,black] (node136) -- (node156);
\draw[line width=0.04mm,black] (node136) -- (node356);
\draw[line width=0.04mm,black] (node156) -- (node126);
\draw[line width=0.04mm,black] (node156) -- (node136);
\draw[line width=0.04mm,black] (node156) -- (node356);
\draw[line width=0.04mm,black] (node156) -- (node456);
\draw[line width=0.04mm,black] (node234) -- (node123);
\draw[line width=0.04mm,black] (node234) -- (node235);
\draw[line width=0.04mm,black] (node234) -- (node345);
\draw[line width=0.04mm,black] (node235) -- (node123);
\draw[line width=0.04mm,black] (node235) -- (node135);
\draw[line width=0.04mm,black] (node235) -- (node234);
\draw[line width=0.04mm,black] (node235) -- (node345);
\draw[line width=0.04mm,black] (node235) -- (node356);
\draw[line width=0.04mm,black] (node345) -- (node234);
\draw[line width=0.04mm,black] (node345) -- (node235);
\draw[line width=0.04mm,black] (node345) -- (node356);
\draw[line width=0.04mm,black] (node345) -- (node456);
\draw[line width=0.04mm,black] (node356) -- (node135);
\draw[line width=0.04mm,black] (node356) -- (node136);
\draw[line width=0.04mm,black] (node356) -- (node156);
\draw[line width=0.04mm,black] (node356) -- (node235);
\draw[line width=0.04mm,black] (node356) -- (node345);
\draw[line width=0.04mm,black] (node356) -- (node456);
\draw[line width=0.04mm,black] (node456) -- (node156);
\draw[line width=0.04mm,black] (node456) -- (node345);
\draw[line width=0.04mm,black] (node456) -- (node356);
\node[draw,ellipse,black,fill=white] (node123) at (2.80,5.74) {$1\,2\,3$};
\node[draw,ellipse,black,fill=white] (node126) at (5.60,3.92) {$1\,2\,6$};
\node[draw,ellipse,black,fill=white] (node135) at (-0.00,3.92) {$1\,3\,5$};
\node[draw,ellipse,black,fill=white] (node136) at (2.80,3.92) {$1\,3\,6$};
\node[draw,ellipse,black,fill=white] (node156) at (2.80,2.10) {$1\,5\,6$};
\node[draw,ellipse,black,fill=white] (node234) at (-2.80,5.74) {$2\,3\,4$};
\node[draw,ellipse,black,fill=white] (node235) at (-1.40,4.83) {$2\,3\,5$};
\node[draw,ellipse,black,fill=white] (node345) at (-5.60,3.92) {$3\,4\,5$};
\node[draw,ellipse,black,fill=white] (node356) at (-1.40,3.01) {$3\,5\,6$};
\node[draw,ellipse,black,fill=white] (node456) at (-2.80,2.10) {$4\,5\,6$};
\end{tikzpicture}}
\\\hline\\
\scalebox{0.8}{
\begin{tikzpicture}
\node[draw,ellipse,black,fill=white] (node123) at (2.80,5.74) {$1\,2\,3$};
\node[draw,ellipse,black,fill=white] (node126) at (5.60,3.92) {$1\,2\,6$};
\node[draw,ellipse,black,fill=white] (node135) at (-0.00,3.92) {$1\,3\,5$};
\node[draw,ellipse,black,fill=white] (node136) at (2.80,3.92) {$1\,3\,6$};
\node[draw,ellipse,black,fill=white] (node156) at (2.80,2.10) {$1\,5\,6$};
\node[draw,ellipse,black,fill=white] (node234) at (-2.80,5.74) {$2\,3\,4$};
\node[draw,ellipse,black,fill=white] (node235) at (-1.40,4.83) {$2\,3\,5$};
\node[draw,ellipse,black,fill=white] (node345) at (-5.60,3.92) {$3\,4\,5$};
\node[draw,ellipse,black,fill=white] (node356) at (-1.40,3.01) {$3\,5\,6$};
\node[draw,ellipse,black,fill=white] (node456) at (-2.80,2.10) {$4\,5\,6$};
\fill [opacity=0.2,black] (node123.center) -- (node126.center) -- (node136.center) -- cycle;
\fill [opacity=0.2,black] (node123.center) -- (node135.center) -- (node235.center) -- cycle;
\fill [opacity=0.2,black] (node135.center) -- (node136.center) -- (node156.center) -- (node356.center) -- cycle;
% \fill [opacity=0.2,black] (node136.center) -- (node156.center) -- (node356.center) -- cycle;
\fill [opacity=0.2,black] (node234.center) -- (node235.center) -- (node345.center) -- cycle;
\fill [opacity=0.2,black] (node345.center) -- (node356.center) -- (node456.center) -- cycle;
\draw[line width=0.04mm,black] (node123) -- (node126);
\draw[line width=0.04mm,black] (node123) -- (node135);
\draw[line width=0.04mm,black] (node123) -- (node136);\draw[line width=0.06mm,black] (node123.center) to[bend right=5] (node234.center)  to[bend right=5] (node123.center) ;
\draw[line width=0.06mm,black,fill=black,opacity=0.2] (node123.center) to[bend right=5] (node234.center)  to[bend right=5] (node123.center) ;
\draw[line width=0.04mm,black] (node123) -- (node235);
\draw[line width=0.04mm,black] (node126) -- (node123);
\draw[line width=0.04mm,black] (node126) -- (node136);
\draw[line width=0.04mm,black] (node135) -- (node123);
\draw[line width=0.04mm,black] (node135) -- (node136);
\draw[line width=0.04mm,black] (node135) -- (node235);
\draw[line width=0.04mm,black] (node135) -- (node356);
\draw[line width=0.04mm,black] (node136) -- (node123);
\draw[line width=0.04mm,black] (node136) -- (node126);
\draw[line width=0.04mm,black] (node136) -- (node135);
\draw[line width=0.04mm,black] (node136) -- (node156);
% \draw[line width=0.04mm,black] (node136) -- (node356);
\draw[line width=0.06mm,black] (node156.center) to[bend right=5] (node126.center)  to[bend right=5] (node156.center) ;
\draw[line width=0.06mm,black,fill=black,opacity=0.2] (node156.center) to[bend right=5] (node126.center)  to[bend right=5] (node156.center) ;
\draw[line width=0.04mm,black] (node156) -- (node136);
\draw[line width=0.04mm,black] (node156) -- (node356);
\draw[line width=0.04mm,black] (node234) -- (node235);
\draw[line width=0.04mm,black] (node234) -- (node345);
\draw[line width=0.04mm,black] (node235) -- (node123);
\draw[line width=0.04mm,black] (node235) -- (node135);
\draw[line width=0.04mm,black] (node235) -- (node234);
\draw[line width=0.04mm,black] (node235) -- (node345);\draw[line width=0.06mm,black] (node235.center) to[bend right=5] (node356.center)  to[bend right=5] (node235.center) ;
\draw[line width=0.06mm,black,fill=black,opacity=0.2] (node235.center) to[bend right=5] (node356.center)  to[bend right=5] (node235.center) ;
\draw[line width=0.04mm,black] (node345) -- (node234);
\draw[line width=0.04mm,black] (node345) -- (node235);
\draw[line width=0.04mm,black] (node345) -- (node356);
\draw[line width=0.04mm,black] (node345) -- (node456);
\draw[line width=0.04mm,black] (node356) -- (node135);
% \draw[line width=0.04mm,black] (node356) -- (node136);
\draw[line width=0.04mm,black] (node356) -- (node156);
\draw[line width=0.04mm,black] (node356) -- (node345);
\draw[line width=0.04mm,black] (node356) -- (node456);\draw[line width=0.06mm,black] (node456.center) to[bend right=5] (node156.center)  to[bend right=5] (node456.center) ;
\draw[line width=0.06mm,black,fill=black,opacity=0.2] (node456.center) to[bend right=5] (node156.center)  to[bend right=5] (node456.center) ;
\draw[line width=0.04mm,black] (node456) -- (node345);
\draw[line width=0.04mm,black] (node456) -- (node356);
\node[draw,ellipse,black,fill=white] (node123) at (2.80,5.74) {$1\,2\,3$};
\node[draw,ellipse,black,fill=white] (node126) at (5.60,3.92) {$1\,2\,6$};
\node[draw,ellipse,black,fill=white] (node135) at (-0.00,3.92) {$1\,3\,5$};
\node[draw,ellipse,black,fill=white] (node136) at (2.80,3.92) {$1\,3\,6$};
\node[draw,ellipse,black,fill=white] (node156) at (2.80,2.10) {$1\,5\,6$};
\node[draw,ellipse,black,fill=white] (node234) at (-2.80,5.74) {$2\,3\,4$};
\node[draw,ellipse,black,fill=white] (node235) at (-1.40,4.83) {$2\,3\,5$};
\node[draw,ellipse,black,fill=white] (node345) at (-5.60,3.92) {$3\,4\,5$};
\node[draw,ellipse,black,fill=white] (node356) at (-1.40,3.01) {$3\,5\,6$};
\node[draw,ellipse,black,fill=white] (node456) at (-2.80,2.10) {$4\,5\,6$};
\end{tikzpicture}}
\\\hline
\end{tabular}
\caption{\label{fig:tiling}A maximal {by inclusion} chord separated collection in $[6]\choose 3$ corresponds to a plabic tiling (top). A triangulated plabic tiling (middle). A modified plabic tiling $\TPTiling_i'$ from the proof of Lemma~\ref{lemma:lifting} (bottom).}
\end{figure}

% \subsection{Triangulated plabic tilings}\label{subsect:TPTilings}
We now introduce a new object that is dual to reduced trivalent plabic graphs from Section~\ref{sect:plabic_graphs}:
\begin{definition}
 A \emph{triangulated plabic tiling}  $\TPTiling$ corresponding to a maximal {by inclusion} chord separated collection $\WS\subset{[n]\choose k}$ is any triangulation of $\PTiling(\WS)$ with the same set of vertices as $\PTiling(\WS)$. In other words, $\TPTiling$ consists of $\PTiling(\WS)$ together with a triangulation of each face of $\PTiling(\WS)$ corresponding to a non-trivial clique of size at least $4$. In this case, $k$ is called the \emph{level} of $\TPTiling$.
\end{definition}

For example, Figure~\ref{fig:tiling} (middle) contains one of the four possible triangulated plabic tilings corresponding to the collection in Figure~\ref{fig:tiling} (top). Note that if an edge of a triangulated plabic tiling connects $\v_S$ and $\v_T$ then $S=Tj- i$ and $T=Si- j$ for some $i\neq j$.

Each triangulated plabic tiling $\TPTiling$ has the underlying chord separated collection $\WS$, and the vertices of $\TPTiling$ are naturally labeled by the elements of $\WS$. We denote by $\Vert(\TPTiling)$ the collection of such labels, thus $\WS=\Vert(\TPTiling)$ can be reconstructed from $\TPTiling$ in a trivial manner.

\begin{lemma}\label{lemma:lifting}
 If $\TPTiling_i$ is a triangulated plabic tiling of level $i$ then the collection 
 \[\UP(\TPTiling_i):=\{S\cup T\mid (\v_S,\v_T)\text{ is an edge of $\TPTiling_i$}\}\]
 has size $(i+1)(n-i-1)+1$. Similarly, the collection
 \[\DOWN(\TPTiling_i):=\{S\cap T\mid (\v_S,\v_T)\text{ is an edge of $\TPTiling_i$}\}\]
 has size $(i-1)(n-i+1)+1$.
\end{lemma}

\begin{proof}
The relation of chord separation is preserved under taking complements, so we only need to prove the first claim.

Take $\TPTiling_i$ and consider a pair of white triangles sharing an edge. Replace this edge by a pair of parallel edges and a black $2$-gon inside. Do this for all such pairs, and also do the same thing for each boundary edge adjacent to a white triangle. For each pair of black triangles sharing an edge, remove that edge. Let $\TPTiling_i'$ be the resulting modified tiling, which is a planar graph with faces colored black and white. Each of its non-boundary edges lies between a white face and a black face. Such a modified tiling $\TPTiling_i'$ is illustrated in Figure~\ref{fig:tiling}(bottom).

The sets in $\UP(\TPTiling_i)$ correspond naturally to the black faces of $\TPTiling'_i$. All the white faces of $\TPTiling'_i$ are triangles, and thus the number of edges in $\TPTiling'_i$ equals $3w+n$, where $w$ is the number of white faces in $\TPTiling'_i$ and $n$ is the number of boundary edges in $\TPTiling'_i$ (by construction, all of them are adjacent to a black face). Let $b$ be the number of black faces of $\TPTiling'_i$, so $|\UP(\TPTiling_i)|=b$ is the number that we want to find. By Theorem~\ref{thm:plabic_OPS}, we know that the number of vertices of $\TPTiling'_i$ equals $i(n-i)+1$. 

\def\Le{{ \reflectbox{$\mathrm{L}$}}}

It follows from the general theory of plabic graphs (see~\cite{Postnikov}) that the number of white triangles in $\TPTiling'_i$ equals $w=i(n-i-1)$. Indeed, $w$ is the number of trivalent white vertices in a reduced plabic graph corresponding to the top cell, and this number is invariant under square moves and contractions/uncontractions, so one just needs to count this number for the corresponding \emph{$\Le$ diagram}, which in this case is a rectangle, see \cite[Figure~20.1]{Postnikov}.

We can now apply Euler's formula, where we also need to count the outer face:
\[i(n-i)+1-(3w+n)+(w+b+1)=2.\]
This is an equation on $b$ yielding the desired result $b=(i+1)(n-i-1)+1$. As an example, take the tiling from the bottom of Figure~\ref{fig:tiling}. We have $i=3,n=6,w=6,b=9$ so the above equation is indeed satisfied:
\[3\cdot 3+1-(3\cdot 6+6)+(6+9+1)=10-24+16=2.\]
\end{proof}

\begin{definition}\label{dfn:compatible}
 Let $\TPTiling_i,\TPTiling_{i+1}$ be any two triangulated plabic tilings of levels $i$ and $i+1$ respectively. We say that the triangulated plabic tilings $\TPTiling_i,\TPTiling_{i+1}$ are \emph{compatible} if 
\begin{itemize}
 \item for every edge $(\v_S,\v_T)$ of $\TPTiling_i$, $\v_{S\cup T}$ is a vertex of $\TPTiling_{i+1}$;
 \item for every edge $(\v_S,\v_T)$ of $\TPTiling_{i+1}$, $\v_{S\cap T}$ is a vertex of $\TPTiling_i$.
\end{itemize}
\end{definition}

Lemma~\ref{lemma:lifting} combined with Theorem~\ref{thm:plabic_OPS} yields the following:
\begin{corollary}\label{cor:lifting}
 Let $\TPTiling_i,\TPTiling_{i+1}$ be two compatible triangulated plabic tilings of levels $i$ and $i+1$ respectively. Then 
 \[\Vert(\TPTiling_{i+1})=\UP(\TPTiling_i);\quad \Vert(\TPTiling_i)=\DOWN(\TPTiling_{i+1}).\]
\end{corollary}
\begin{proof}
 By definition, the collection $\Vert(\TPTiling_{i+1})$ contains $\UP(\TPTiling_i)$. By Lemma~\ref{lemma:lifting}, their sizes coincide.
\end{proof}

By definition, the collection $\Vert(\TPTiling_i)$ is chord separated (and is maximal {by inclusion} in $[n]\choose i$). It turns out that given two compatible tilings, their vertices are chord separated from each other as well:
\begin{lemma}\label{lemma:compatible_levels}
 Let $\TPTiling_i,\TPTiling_{i+1}$ be two compatible triangulated plabic tilings of levels $i$ and $i+1$ respectively. Then the collection $\Vert(\TPTiling_i)\cup\Vert(\TPTiling_{i+1})$ is chord separated.
\end{lemma}
\begin{proof}
 Suppose that $T\in\Vert(\TPTiling_i)$ and $S\in\Vert(\TPTiling_{i+1})$ are not chord separated, that is, there exist cyclically ordered integers $a,b,c,d\in [n]$ such that $a,c\in S-T, b,d\in T-S$. By Corollary~\ref{cor:lifting}, there is an edge in $\TPTiling_{i+1}$ connecting $Tx$ and $Ty$ for some $x\neq y\not\in T$. Both of these sets have to be chord separated from $S$ and the only way for this to be possible is if $x=a$ and $y=c$ (or vice versa). Note that $|Txy|=i+2$ and $|S|=i+1$, but $b,d\in Txy-S$. It means that $S\not\subset Txy$, so let $z\in S-Txy$. And because $x=a$ and $y=c$, we have $z\neq a,b,c,d$. If $z$ belongs to the cyclic interval $(d,b)$ then $S$ is not chord separated from $Tx=Ta$. Otherwise $z$ belongs to the cyclic interval $(b,d)$ in which case $S$ is not chord separated from $Ty=Tc$. This contradiction finishes the proof of the lemma.
\end{proof}

A family $\TPTFamily:=(\TPTiling_0,\TPTiling_1,\dots,\TPTiling_n)$ of triangulated plabic tilings such that the level of $\TPTiling_i$ is equal to $i$ and $\TPTiling_i$ and $\TPTiling_{i+1}$ are compatible for all $i=0,\dots,n-1$ is called \emph{admissible}.

 Now we go back to collections $\WS\subset 2^{[n]}$ without any restrictions on sizes. Given such a collection, for all $i=0,1,\dots, n$ denote $\WS^\ipar:=\WS\cap{[n]\choose i}$. For an admissible family $\TPTFamily:=(\TPTiling_0,\TPTiling_1,\dots,\TPTiling_n)$ of triangulated plabic tilings, denote 
 \[\Vert(\TPTFamily)=\cup_{i=0}^n\Vert(\TPTiling_i)\] 
 to be the set of all vertex labels of tilings in $\TPTFamily$. 
 
 \begin{proposition}\label{prop:families_collections}
 The map $\TPTFamily\mapsto\Vert(\TPTFamily)$ is a bijection between admissible families of triangulated plabic tilings and maximal {by inclusion} chord separated collections in $2^{[n]}$. In particular, if $\WS\subset 2^{[n]}$ is any maximal {by inclusion} chord-separated collection then  $\WS^\ipar\subset {[n]\choose i}$  is also maximal {by inclusion} (and by size), and if $\WS=\Vert(\TPTFamily)$ then 
 \[\WS^\ipar=\Vert(\TPTiling_i).\]
 \end{proposition}

Before we proceed to the proof, let us explain the connection between admissible families of triangulated plabic tilings and zonotopal tilings.

Consider the zonotope $\Zon(n,3)$ and its fine zonotopal tiling $\Tiling$. Recall that for each $i=0,\dots,n$, $\Hyp_i$ is the plane given by $z=i$. Denote $\Zon(n,3)^\ipar:=\Zon(n,3)\cap \Hyp_i$ and $\Tiling^\ipar:=\{\Face_X\cap \Hyp_i\mid X\in\Tiling\}.$ It is clear that $\Zon(n,3)^\ipar$ is a convex polygon and  $\Tiling^\ipar$ is a triangulation of $\Zon(n,3)^\ipar$ (which may have some vertices inside $\Zon(n,3)^\ipar$). Moreover, by definition, each vertex of $\Tiling$, and thus of $\Tiling^\ipar$, is labeled by a subset of $[n]$. 

\begin{proposition}\label{prop:tilings_families}
 The map $\Tiling\mapsto (\Tiling^\ipar)_{i=0}^n$ is a bijection between fine zonotopal tilings of $\Zon(n,3)$ and admissible families of triangulated plabic tilings.
\end{proposition}
\begin{proof}
First we show that $(\Tiling^\ipar)_{i=0}^n$ is an admissible family of triangulated plabic tilings. Let $\WS:=\Vert(\Tiling)$ and suppose that $\WS$ is not a chord separated collection. Then there must exist integers $a<b<c<d\in [n]$ and two sets $S,T\in\WS$ such that $a,c\in S-T$ and $b,d\in T-S$. Recall from Section~\ref{sect:sections} that fine zonotopal tilings of $\Zon(n,3)$ are dual to arrangements of pseudoplanes. Let $(\PSP_1,\dots,\PSP_n)$ be the pseudoplane arrangement dual to $\Tiling$. Since $(\{a,c\},\{b,d\})$ is a Radon partition of of $\Cyclic(n,3)$, the pseudoplanes $P_a,P_b,P_c,P_d$ bound a certain simplicial region $R$ in $\R^3$, and both vertices $\v_S$ and $\v_T$ have to be inside $R$. This leads to a contradiction because if $S$ and $T$ belong to the same region, their restrictions on $\{a,b,c,d\}$ must also coincide. A more rigorous argument can be given using Ziegler's bijection\footnote{this is basically an application of Bohne-Dress theorem: $\Tiling$ defines a localization in general position of the dual matroid, and this localization is the same thing as $\phi(\Tiling)$ viewed as a map from the cocircuits of the dual matroid to $\{+,-\}$.} $\phi$ between fine zonotopal tilings of $\Zon(n,3)$ and \emph{consistent} subsets of $[n]\choose 4$. One easily observes that under this bijection, $\phi(\Tiling)$ contains $\{a,b,c,d\}$ if and only if there is a subset $U\in\Vert(\Tiling)$ with $a,c\in U$ and $b,d\not\in U$. Similarly, $\phi(\Tiling)$ does not contain $\{a,b,c,d\}$ if and only if there is a subset $U\in\Vert(\Tiling)$ with $b,d\in U$ and $a,c\not\in U$. The fact that $\phi(\Tiling)$ is well defined implies that $S$ and $T$ cannot both appear in $\Vert(\Tiling)$. 

A simple deletion-contraction argument shows that the number of vertices of $\Tiling$ equals the number of independent sets of the corresponding oriented matroid. In other words, the vertex labels $\WS:=\Vert(\Tiling)$ form a chord-separated collection in $2^{[n]}$ satisfying
 \[|\WS|={n\choose 0}+ {n\choose 1}+{n\choose 2}+{n\choose 3}.\]
Note that since $\WS$ is a chord separated collection, $\WS^\ipar$ is a weakly separated collection in ${[n]\choose i}$,
%. By \cite[Theorem~1.3]{OPS}, or \cite[Theorem~A]{DKK10}, or even \cite[Theorem~1.3]{LZ}, every such collection 
and hence has size at most $i(n-i)+1$. Summing over all $i$, we get
\[\sum_{i=0}^n|\WS^\ipar|\leq \sum_{i=0}^n (i(n-i)+1)=\frac{n^3+5n+1}6.\]
One easily checks that this is equal to 
\[{n\choose 0}+ {n\choose 1}+{n\choose 2}+{n\choose 3}=|\WS|.\]
Thus all the inequalities become equalities and each $\WS^\ipar$ is a maximal \emph{by size} weakly separated collection in ${[n]\choose i}$. 

Next, observe that for each top-dimensional zonotope $\Face_X$ in $\Tiling$, the set $X^0$ has exactly three elements, say, $a,b,c\in[n]$, and thus all vertices of $\Face_X$ have the form 
\[\Vert(\Face_X)=\{\v_S,\v_{Sa},\v_{Sb},\v_{Sc},\v_{Sab},\v_{Sbc},\v_{Sac},\v_{Sabc}\}\]
for $S=X^+$. Denote $k:=|S|$, then the only non-empty intersections of $\Face_X$ with planes $\Hyp_i$ happen for $i=k,k+1,k+2,k+3$. Moreover, for $i=k$ and $i=k+3$, the intersection is just a single vertex, and for $i=k+1,k+2$, the intersection is a triangle. For example, for $i=k+1$, this triangle has vertices $\v_{Sa},\v_{Sb},\v_{Sc}$, and thus is a subset of the corresponding white clique $\WhCl(S)$ in $\PTiling(\WS^\ipar)$. Similarly, for $i=k+2$, the triangle lies inside a black clique $\BlCl(Sabc)$ in $\PTiling(\WS^\ipar)$. Thus $(\Tiling^\ipar)_{i=0}^n$ indeed form a family of triangulated plabic tilings. For every edge connecting, for example, $\v_{Sa}$ and $\v_{Sb}$, $\v_{Sab}$ is a vertex on level $k+2$ and $\v_{S}$ is a vertex on level $k$, and similarly, for every edge of the form $(\v_{Sab},\v_{Sbc})$, $\v_{Sb}$ is a vertex on level $k+1$ and $\v_{Sabc}$ is a vertex on level $k+3$. Thus the family $(\Tiling^\ipar)_{i=0}^n$ is admissible, and we have finished one direction of the proposition.

Now assume that we are given an admissible family $\TPTFamily$ of triangulated plabic tilings, and we want to construct a fine zonotopal tiling $\Tiling$ of $\Zon(n,3)$ satisfying $\Tiling^\ipar=\TPTFamily^\ipar$ for all $i=0,1,\dots, n$. It suffices to describe the top-dimensional tiles $\Face_X$ of $\Tiling$. As before, each such tile can be indexed by a subset $S=X^+$ and three elements $a,b,c$ with $\{a,b,c\}=X^0$, so we denote it $\Face_{S,a,b,c}$. We are ready to describe all top-dimensional faces of $\Tiling$: they are all faces of the form $\Face_{S,a,b,c}$ where $Sa,Sb,$ and $Sc$ are the vertex labels of a \emph{white} triangle in one of the tilings in $\TPTFamily$. 

So far, we only know that thus defined $\Tiling$ is a collection of parallelotopes, and we need to see why their union is the whole $\Zon(n,3)$ and why the intersection of any two of them is their common face. We start with the latter.

\begin{claim}
The intersection of any two tiles is in $\Tiling$ is either empty or their common face. 
\end{claim}
\begin{claimproof}
We will only consider the case of two top-dimensional tiles intersecting in their relative interior, the case of lower-dimensional tiles is handled similarly. Consider two tiles $\Face_{S,a,b,c}$ and $\Face_{T,d,e,f}$ and suppose that their interiors intersect on some point $p\in\Face_{S,a,b,c}\cap\Face_{T,d,e,f}$. Then there exist numbers $0<\l_a,\l_b,\l_c,\m_d,\m_e,\m_f<1$ satisfying
\begin{equation}\label{eq:convex_comb}
 \sum_{s\in S} 1\cdot \v_s+\l_a\v_a+\l_b\v_b+\l_c\v_c=\sum_{t\in T} 1\cdot \v_t+\m_d\v_d+\m_e\v_e+\m_f\v_f.
\end{equation}
We are going to construct two subsets $S'\in\Vert(\Face_{S,a,b,c})$ and $T'\in \Vert(\Face_{T,d,e,f})$ that are not chord separated but such that their sizes differ by at most $1$. Let $X=(X^+,X^-)$ be the signed vector defined as follows: $X^+$ contains all elements $g$ of $[n]$ such that the coefficient of $\v_g$ in the left hand side of~(\ref{eq:convex_comb}) is strictly greater than the coefficient of $\v_g$ in the right hand side of~(\ref{eq:convex_comb}). The set $X^-$ is defined similarly, replacing left with right. Since the left hand side equals the right hand side, there exist cyclically ordered integers $p,q,r,s$ such that $C^+:=\{p,r\}\subset X^+$ and $C^-:=\{q,s\}\subset X^-$ (such a signed set $C:=(C^+,C^-)$ is called a \emph{Radon partition} of $\Cyclic(n,3)$). 

It is clear that $C^+\subset Sabc-T$ while $C^-\subset Tdef-S$. Thus there exists at least one pair of subsets $S'\in\Face_{S,a,b,c}$ and $T'\in \Face_{T,d,e,f}$ such that $C^+\subset S'-T'$ and $C^-\subset T'-S'$. We want to choose $S'$ and $T'$ so that their levels differ by at most $1$. Suppose this cannot be done, and that the level of any such $S'$ is, say, greater than the level of any such $T'$ by at least $2$. Then we set $S':=S\cup C^+$ and $T':=Tdef-C^+$. These two sets clearly satisfy $C^+\subset S'-T'$ and $C^-\subset T'-S'$, so if we show that $|S'|\leq |T'|+1$ then we will get a contradiction.

Let $h$ be the sum of the coefficients in any of the sides of~(\ref{eq:convex_comb}) (they are equal since every vector from $\Cyclic(n,3)$ has height $1$). If $h$ is an integer then the bad intersection would happen inside one triangulated plabic tiling which is impossible by Theorem~\ref{thm:plabic_OPS}. Thus let $h^-<h<h^+$ be the floor and the ceiling of $h$. 

We are going to consider a number of cases:
\begin{enumerate}
 \item\label{item:caseSTclose} $|S'|\leq h^+$ and $|T'|\geq h^-$;
 \item \label{item:caseTsmall} $|T'|<h^-$;
 \item \label{item:caseSbig} $|S'|>h^+$.
\end{enumerate}

For each of them, we need to show that $|S'|\leq |T'|+1$. Case~(\ref{item:caseSTclose}) implies that trivially. Cases (\ref{item:caseTsmall}) and (\ref{item:caseSbig}) are symmetric with respect to taking complements of all sets, so we only need to consider one of them. From now on, we assume that $|S'|>h^+$. Recall that $S'=S\cup C^+$ where $C^+$ is a two-element set. Thus $|S'|\leq |S|+2$. Now, the integers $|S|, h^+,$ and $|S'|$ satisfy (since $|S|\leq h$ and $h<h^+$)
\[|S|<h^+<|S'|\leq |S|+2\]
which implies that 
\[|S|=h^-;\quad |S|+1=h^+;\quad |S|+2=|S'|.\]
In particular, the last equality implies that $C^+\cap S=\emptyset$, in other words, that $C^+\subset \{a,b,c\}$. Without loss of generality we may assume that $C^+=\{a,b\}$. It follows that $\l_a+\l_b<h-|S|$ because
\[h=|S|+\l_a+\l_b+\l_c>|S|+\l_a+\l_b.\]
% 
% 
% 
% 
% If $|S'|\leq h^+$ and $|T'|\geq h^-$ then we are done. Suppose that $|S'|>h^+$. This means that $|C^+\cap \{a,b,c\}|=2$, because we know that $|C^+|=2$ and if the intersection would have size $1$ or $0$ then the level of $S'$ would be at most $h^+$: indeed, suppose $C^+\cap\{a,b,c\}=\{a\}$, then the coefficient of $a$ in the left hand side is positive, and thus $h>|S|$ while 
% $|S|+1=|S'|\leq h^+$. So assume that $|S'|=h^++1, |S|=h^-$, $C^+\cap\{a,b,c\}=\{a,b\}$, and $\l_a+\l_b<1$. The last inequality is true since 
% \[|S'|=|S|+2> h+1\geq |S|+\l_a+\l_b+1.\]

Our goal is to show that $|T'|\geq h^+$. We are going to consider three cases:
\begin{enumerate}[(a)]
 \item \label{item:Cminus2} $|C^+\cap \{d,e,f\}|=2$;
 \item \label{item:Cminus1} $|C^+\cap \{d,e,f\}|=1$;
 \item \label{item:Cminus0} $|C^+\cap \{d,e,f\}|=0$;
\end{enumerate}

First consider case~(\ref{item:Cminus2}), say, $d=a$ and $e=b$. By definition, we then have $|T'|=|T|+1$. Since $a,b\in C^+\subset X^+$, the coefficients of $\v_a$ and $\v_b$ in the right hand side are less than the corresponding coefficients in the left hand side. We get that 
\[\m_d+\m_e<\l_a+\l_b<h-|S|.\]
But this implies that $|T|\geq |S|$, because otherwise if $|T|\leq |S|-1$ then 
\[h=|T|+\m_d+\m_e+\m_f< |S|-1+(h-|S|)+\m_f< h,\]
which is impossible. Thus $|T|\geq |S|=h^-$ and since $|T'|=|T|+1$, we get that $|T'|\geq h^+$. We are done with case~(\ref{item:Cminus2}). 

Now consider case~(\ref{item:Cminus1}), say, $d=a$. Then by definition, $|T'|= |T|+2$. Also, since $a\in C^+\subset X^+$, we have $\m_d<\l_a<h-|S|$. We need to show $|T'|\geq h^+$, equivalently, $|T|\geq |S|-1$. This is indeed true, because otherwise if $|T|\leq |S|-2$ then 
\[h=|T|+\m_d+\m_e+\m_f< |S|-2+(h-|S|)+\m_e+\m_f< h,\]
which is impossible. Thus we are done with the case~(\ref{item:Cminus1}), and the case~(\ref{item:Cminus0}) is trivial since for this case, $|T'|=|T|+3>h$ and so $|T'|\geq h^+$. We have finished the proof of the fact that $S'$ and $T'$ can be chosen so that their levels differ by at most $1$. But they are not chord separated which contradicts Lemma~\ref{lemma:compatible_levels}.
\end{claimproof}

\begin{claim}
 The union $U$ of all tiles in $\Tiling$ equals $\Zon(n,3)$.
\end{claim}
\begin{claimproof}
Observe that $U$ is a closed polyhedral complex inside $\Zon(n,3)$ that contains the boundary of $\Zon(n,3)$, because the vertices of the boundary are labeled by all cyclic intervals.
% , that is, sets of the form 
% \[[i,j]:=\begin{cases}
% \{i,i+1,\dots,j\}, &\text{ if $i\leq j$};\\
% \{i,i+1,\dots,n,1,2,\dots,j\}, &\text{ if $i>j$.}                                                                                                                                                                                                                                                                                                                   \end{cases}\]

Assume that $U\subsetneq \Zon(n,3)$ and let $p\in\Zon(n,3)- U$ be any point. Choose a generic point $q$ in the interior of any tile of $\Tiling$, and draw a segment $[p,q]$. Let $\Face_X$ be the closest to $p$ $2$-dimensional face of $\Tiling$ that intersects $[p,q]$, and let $r$ be their intersection point. Such a face has vertices $\v_S,\v_{Sa},\v_{Sb},\v_{Sab}$ for some $S\subset [n]$ and $a,b\in [n]$. But then there is an edge between $\v_{Sa}$ and $\v_{Sb}$ in the corresponding triangulated plabic tiling $\TPTiling_k$ where $k=|S|+1$. This edge is the common edge of two triangles, and each of them belongs to a three-dimensional tile, so we can see that $\Face_X$ is the common face of these two tiles. Thus $r$ is contained in $U$ together with  some open neighborhood, which leads to a contradiction, thus finishing the proof or  the claim.
\end{claimproof}
Proposition~\ref{prop:tilings_families} follows from the two claims above.
\end{proof}

\begin{proof}[Proof of Proposition~\ref{prop:families_collections}]
 As we have shown in the proof of Proposition~\ref{prop:tilings_families}, if $\TPTFamily$ is an admissible family of triangulated plabic tilings then $\Vert(\TPTFamily)$ is a maximal \emph{by size} chord separated collection in $2^{[n]}$. In particular, it is also maximal \emph{by inclusion}, which shows one direction of the proposition. Assume now that we are given a maximal \emph{by inclusion} chord separated collection $\WS\subset 2^{[n]}$, and we need to reconstruct an admissible family $\TPTFamily$ with $\Vert(\TPTFamily)=\WS$. Note that it is not even a priori clear why $\WS^\ipar$ is maximal by inclusion in ${[n]\choose i}$. But we claim that there exists a unique admissible family $\TPTFamily$ such that:
 \begin{itemize}
  \item the vertex labels of $\TPTFamily$ are precisely the sets in $\WS$;
  \item $(\v_S,\v_T)$ is an edge of $\TPTiling_i$ if and only if 
  \begin{equation}\label{eq:edges}
   |S|=|T|=i,\  |S\cap T|=i-1,\ |S\cup T|=i+1,\ \text{and } S\cap T,S\cup T\in\WS. 
  \end{equation}
 \end{itemize}

 We are going to reconstruct the triangulated plabic tilings $\TPTiling_i$ one by one for $i=0,1,\dots, n$. For $i=0$, $\TPTiling_0$ consists of the unique vertex labeled by $\emptyset$. For $i=1$, all the one-element sets have to belong to $\WS$ because it is maximal \emph{by inclusion} and one-element sets are chord separated from all other sets. Thus $\WS^{(1)}$ is a maximal \emph{by inclusion} collection in ${[n]\choose 1}$. Our induction on $i$ is going to be based on the following lemma:
 
 \begin{lemma}\label{lemma:triangulation}
  Suppose that $\WS\subset 2^{[n]}$ is a maximal \emph{by inclusion} chord separated collection, and that there exists an $i\in[n]$ such that $\WS^\ipar$ is a maximal \emph{by inclusion} collection in ${[n]\choose i}$, so that it corresponds to a plabic tiling $\PTiling(\WS^\ipar)$ via Theorem~\ref{thm:plabic_OPS}. Then (\ref{eq:edges}) gives a triangulation of the polygons of $\PTiling(\WS^\ipar)$, therefore transforming it into a triangulated plabic tiling $\TPTiling_i$.
 \end{lemma}

Let us first show that Lemmas~\ref{lemma:triangulation} and~\ref{lemma:lifting} together imply Proposition~\ref{prop:families_collections}. As we have already shown, $\WS^{(1)}$ is a maximal \emph{by inclusion}  chord separated collection in ${[n]\choose 1}$. Then  Lemma~\ref{lemma:triangulation} gives a unique triangulated plabic tiling $\TPTiling_1$. But then Lemma~\ref{lemma:lifting} states that $\WS^{(2)}\supset \UP(\TPTiling_1)$ is a maximal \emph{by size} (and thus by inclusion) chord separated collection in ${[n]\choose 2}$, after which Lemma~\ref{lemma:triangulation} produces $\TPTiling_2$. Continuing in this fashion, we reconstruct the whole family $\TPTFamily$, and it is compatible by construction, while, clearly, $\Vert(\TPTFamily)=\WS$, which finishes the proof of Proposition~\ref{prop:families_collections}.

\begin{proof}[Proof of Lemma~\ref{lemma:triangulation}]
Let $\WS\subset 2^{[n]}$ and $\WS^\ipar\subset {[n]\choose i}$ be maximal \emph{by inclusion} chord separated collections. We want to show that (\ref{eq:edges}) gives a triangulation of $\PTiling(\WS^\ipar)$. We are only going to show that it gives a triangulation of white cliques of $\PTiling(\WS^\ipar)$, for black cliques, one can just replace all sets in $\WS$ by their complements and then apply the statement for white cliques.

Our first goal is to show that every edge $(\v_S,\v_T)$ of $\PTiling(\WS^\ipar)$ satisfies~(\ref{eq:edges}).

\begin{claim}
 Suppose $\WS\subset 2^{[n]}$ and $\WS^\ipar$ are maximal \emph{by inclusion} chord separated collections, and assume that $\WhCl(S)$ is a non-trivial white clique in $\PTiling(\WS^\ipar)$. Then $S\in\WS$.
\end{claim}
\begin{claimproof}
 Let $\WhCl(S)=\{Sa_1,Sa_2,\dots,Sa_r\}$ for some $r\geq 3$. Suppose we have found $T\in\WS$ and $a,b,c,d$ cyclically ordered such that $a,c\in T-S$ and $b,d\in S-T$. Since $r\geq 3$, we can find an index $j\in [r]$ such that $a_j\neq a,c$. Then $a,c\in T-Sa_j$ and $b,d\in Sa_j-T$ so $Sa_j$ and $T$ are not chord separated which is impossible since they both belong to $\WS$.
\end{claimproof}

\begin{claim}
 Suppose $\WS\subset 2^{[n]}$ and $\WS^\ipar$ are maximal \emph{by inclusion} chord separated collections, and assume that $\BlCl(S)$ is a non-trivial black clique in $\PTiling(\WS^\ipar)$. Then $S\in\WS$.
\end{claim}
\begin{claimproof}
 Follows from the previous claim by replacing all subsets with their complements.
\end{claimproof}

From these two claims, it becomes clear why all edges of $\PTiling$ satisfy~(\ref{eq:edges}): each of them is either a boundary edge, in which case it connects two cyclic intervals so the result follows trivially, or it is an edge that separates a white clique from a black clique, and thus the intersection and the union of the vertex labels belong to $\WS$ by the above claims.

So all edges of $\PTiling(\WS^\ipar)$ satisfy~(\ref{eq:edges}). Add all the other edges given by (\ref{eq:edges}) to it to get a new tiling which we denote $\TPTiling'$. It is clear that these new edges subdivide the polygons of $\PTiling(\WS^\ipar)$ into smaller polygons and do not intersect each other or the edges of $\PTiling(\WS^\ipar)$. 

Let $A:=\{a_1<a_2<\dots<a_r\}\subset [n]$ and put $a_{r+1}:=a_1$. Suppose $S$ is an $i-1$-element subset of $[n]-A$ such that:
\begin{itemize}
 \item $\{Sa_1,\dots,Sa_r\}\subset \WhCl(S)\subset \WS^\ipar$;
 \item $\{Sa_1,\dots,Sa_r\}$ are the vertex labels of a (white) polygon $P$ of $\TPTiling'$;
 \item $(Sa_i,Sa_{i+1})$ is an edge of $P$ for every $i=1,2,\dots,r$;
 \item $P$ has no other edges inside.
\end{itemize}

We assume that $r>3$ and we would like to see why the polygon $P$ actually has to have an edge inside. 

We are going to start with removing all the irrelevant elements. Define a new collection $\WS'\subset 2^{[r]}$ as follows:
\[\WS':=\{T'\subset [r]\mid S\cup \{a_j\}_{j\in T'}\in\WS\}.\]
% \[\WS':=\{\proj(T)\mid S\subset T\text{ and }  T\in\WS\}.\]
\begin{claim}
  The collection $\WS'\subset 2^{[r]}$ is a maximal \emph{by inclusion} chord separated collection. 
\end{claim}

\begin{claimproof}
For each set $R\in\WS$ containing $S$, define $\proj(R)\subset [r]$ by
\[\proj(R):=\{j\in[r]\mid a_j\in R\}.\]
Now we define another collection of sets
\[\WS'':=\{\proj(R)\mid S\subset R\text{ and }  R\in\WS\}.\]

It is clear that $\WS'$ and $\WS''$ are chord separated and $\WS'\subset\WS''$. We are going to show that if a set $T'\subset [r]$ is chord separated from all $\WS''$ then $T'\in\WS'$. If we manage to do this, it will follow that $\WS'=\WS''$ are both maximal \emph{by inclusion} chord separated collections.

Suppose that there exists a set $T'\subset [r]$ which is chord separated from all elements of $\WS''$ but $T:=S\cup \{a_j\}_{j\in T'}$ is not chord separated from at least one set $R\in\WS$. Let $a,b,c,d$ be cyclically ordered elements of $[n]$ such that $a,c\in R-T$ and $b,d\in T-R$ . Our ultimate goal is to show that $S\subset R$ and that $\proj(R)$ and $T'$ are not chord separated. 

Since $T\supset S$, we have $a,c,\not\in S$. Next, $b,d\in A$: otherwise, if at least one of them, say, $b$, is not in $A$ (and therefore is in $S$ because $S\subset T\subset S\cup A$) then $b,d\in Sd-R$ and $a,c\in R-Sd$ which contradicts the fact that both of them belong to $\WS$. Thus $b,d\in A$, so $Sb,Sd\in\WS$. Now assume that there is an element $s\in S-R$. Then $s\neq a,c$ and so $s$ belongs to one of the circuilar intervals: either $s\in (a,c)$ or $s\in (c,a)$. We know that $b\in (a,c)$ and $d\in (c,a)$ and thus if $s\in (a,c)$ then $Sd$ is not chord separated from $R$, and if $s\in (c,a)$ then $Sb$ is not chord separated from $R$, and in any case we get a contradiction, which means $S\subset R$ and we are half way there. 

If $a,c\in A$ then we are done, because $\proj(R)$ and $T'$ would obviously not be chord separated in this case. Assume that $a\not\in A$ and let $j\in [r]$ be such that $a$ belongs to the cyclic interval $(a_j,a_{j+1})$. Since $b,d\in A$ are strictly between $a$ and $c$, we have that $c\in (a_{j+1},a_j)$. The set $Sa_ja_{j+1}$ belongs to $\WS$, and $c\in R-Sa_ja_{j+1}$, which means that either $a_j$ or $a_{j+1}$ belongs to $R$. Suppose this is $a_j$, and then $a_j\neq b,d$ so we can just replace $a$ with $a_j$ while still preserving the fact that $a,b,c,d$ are cyclically ordered and $a,c\in R-T$ while $b,d\in T-R$. After we do the same with $c$, we get that $a,c\in A$ and thus $\proj(R)$ and $T'$ are not chord separated, which finishes the proof of the claim.
\end{claimproof}

Our situation looks as follows. We have a maximal \emph{by inclusion} chord separated collection $\WS'\subset 2^{[r]}$. It clearly contains all the cyclic intervals, including all the one-element sets. In order to complete the proof of Lemma~\ref{lemma:triangulation}, we need to show that $\WS'$ contains at least one non-trivial two-element set, where \emph{non-trivial} means that it is not a cyclic interval.

For a set $T\subset [r]$, define 
\begin{equation}\label{eq:maxgap}
 \maxgap(T):=\max_{(a,b)\cap T=\emptyset} |(a,b)|.
\end{equation}

In other words, $\maxgap(T)$ is the maximal size of the distance between two consecutive entries of $T$. Choose a set $T\in\WS'$ satisfying:
\begin{itemize}
 \item $T$ is non-trivial;
 \item there is no other non-trivial $T'\in\WS'$ with $\maxgap(T')>\maxgap(T)$.
\end{itemize}
Let $(a,b)$ be an interval that maximizes the right hand side of~(\ref{eq:maxgap}). Then $a,b\in T$. Since $T$ is non-trivial, it is not the case that $T=[b,a]$ and thus there is an element $c\in (b,a)-T$. Consider all pairs $(T',c')$ satisying
\begin{itemize}
\item $T'\in\WS'$;
 \item $a,b\in T'$;
 \item $(a,b)\cap T'=\emptyset$;
 \item $c'\in (b,a)-T'$.
\end{itemize}
Among all such pairs $(T',c')$, choose the one which minimizes the size of $(b,c')$. Thus $[b,c')\subset T'$ but $c'\not\in T'$. Let $d\in T'$ be such that $(c',d)\cap T'=\emptyset$. Consider the set $\{d,c'-1\}$. Certainly this is a non-trivial subset of $[r]$, because $(c'-1,d)$ contains $c'$ while $(d,c'-1)$ contains either $a$ or $b$. Thus if $\{d,c'-1\}\in \WS'$ then we are done. Assume that this is not the case, then there is a set $R\subset [r]$ such that $d,c'-1\not\in R$ and $e,f\in R$ with $e\in (c'-1,d)$ and $f\in (d,c'-1)$. If $f\not\in T'$ then $R$ and $T'$ are not chord separated. Therefore $(a,b)$ contains no such $f$'s and thus $(a,b)\cap R$ has to be empty. If either $a\not\in R$ or $b\not\in R$ then $\maxgap(R)>\maxgap(T)$ which is impossible since $R$ is non-trivial. Thus $a,b\in R$. But since $c'-1\not\in R$, we get a contradiction with the fact that the size of $(b,c')$ was minimal possible. Thus $\{d,c'-1\}\in\WS'$ and we have found a non-trivial two element subset in $\WS'$. It means that $\WS$ contains the set $Sa_da_{c'-1}$ as well as $S,Sa_d,Sa_{c'-1}$ which means that we have an edge that subdivides the white polygon $P$ into two smaller polygons. This finishes the proof of Lemma~\ref{lemma:triangulation}.
\end{proof}%of lemma

We have proven both lemmas and therefore we are done with Proposition~\ref{prop:families_collections} as well as Theorem~\ref{thm:purity_positive}. 
\end{proof}%of proposition

\bibliographystyle{plain}
\bibliography{chord_separation}

\end{document}